\documentclass[reqno, 11pt]{amsart}
\usepackage[margin = 1.2 in]{geometry}
\usepackage{amsthm,amsmath,amsfonts}
\usepackage{hyperref}
\usepackage{amssymb}
\usepackage{stmaryrd}
\usepackage{eucal, mathrsfs}
\usepackage[libertine]{newtxmath}
\usepackage{todonotes}
\usepackage{listings}
\usepackage{xcolor}

\usepackage[nameinlink]{cleveref}

\theoremstyle{plain}
\newtheorem{theorem}{Theorem}[section]
\newtheorem{proposition}[theorem]{Proposition}
\newtheorem{lemma}[theorem]{Lemma}
\newtheorem{corollary}[theorem]{Corollary}
\newtheorem{definition}[theorem]{Definition}
\newtheorem{question}[theorem]{Question}
\newtheorem{example}[theorem]{Example}
\newtheorem{claim}[theorem]{Claim}
\newtheorem{assumption}[theorem]{Assumption}
\newtheorem{problem}[theorem]{Problem}

\theoremstyle{remark}
\newtheorem{remark}[theorem]{Remark}

\renewcommand{\P}{\mathbb{P}}

\renewcommand{\k}{\mathscr{k}}
\DeclareMathOperator{\Spec}{Spec }
\DeclareMathOperator{\Hom}{Hom}
\DeclareMathOperator{\PGL}{PGL}
\DeclareMathOperator{\Hilb}{Hilb}
\DeclareMathOperator{\I}{\mathcal{I}}

\renewcommand{\O}{\mathcal{O}}

\DeclareMathOperator{\CT}{\mathsf{CT}}
\newcommand{\from}{\colon}
\DeclareMathOperator{\jac}{\mathsf{Jac}}
\DeclareMathOperator{\Thin}{\mathsf{Thin}}
\DeclareMathOperator{\Fat}{\mathsf{Fat}}
\DeclareMathOperator{\spec}{Spec \,}
\DeclareMathOperator{\Bl}{Bl}
\DeclareMathOperator{\Sym}{Sym}
\DeclareMathOperator{\Gm}{\mathbb{G}_{m}}
\DeclareMathOperator{\Z}{\mathbb{Z}}
\DeclareMathOperator{\SQP}{\mathsf{SQP}}
\DeclareMathOperator{\Gr}{Gr}
\DeclareMathOperator{\CH}{CH}
\DeclareMathOperator{\s}{\mathsf{sq}}
\DeclareMathOperator{\BP}{\mathsf{Bpt}}
\DeclareMathOperator{\Dom}{\mathsf{Dom}}
\DeclareMathOperator{\Inf}{\mathsf{Inf}}
\DeclareMathOperator{\Fin}{\mathsf{Fin}}
\DeclareMathOperator{\Lin}{\mathsf{Lin}}
\DeclareMathOperator{\Inc}{\mathsf{Inc}}
\DeclareMathOperator{\lin}{\mathsf{lin}}
\DeclareMathOperator{\inc}{\mathsf{inc}}
\DeclareMathOperator{\Base}{\mathsf{Base}}
\DeclareMathOperator{\hess}{\mathsf{H}}
\DeclareMathOperator{\codim}{codim}
\DeclareMathOperator{\Supp}{Supp}

\DeclareMathOperator{\Sing}{Sing}
\DeclareMathOperator{\m}{\mathfrak{m}}
\DeclareMathOperator{\ord}{ord}
\DeclareMathOperator{\GL}{GL}
\DeclareMathOperator{\length}{length}

\DeclareMathOperator{\ver}{ver}

\DeclareMathOperator{\red}{red}

\author{Anand Deopurkar \& Anand Patel}
\title{Counting 3-uple Veronese Surfaces}
\date{\today}

\begin{document}
\maketitle

\section{Introduction} 

The classical fact that $d+3$ general points
in $d$-dimensional projective space $\P^{d}$ determine a unique rational
normal curve can be seen in many ways:
\begin{itemize}
\item By explicit algebraic construction.
\item By Steiner's geometric construction.
\item By an elementary degeneration argument, wherein $d+1$ of the points specialize into a hyperplane.
\item By an application of Goppa's lemma from the theory of association or Gale duality. 
\end{itemize}

A {\sl $d$-uple Veronese $n$-fold} will mean any
variety in $\P^{{n+d \choose n}-1}$ projectively equivalent to the
standard $d$-uple Veronese image of $\P^n$. A parameter count 
uncovers an infinite array of enumerative problems, with the rational normal curves occupying only the first column: 
\begin{problem}
    \label{problem:main}
Determine the number of $d$-uple Veronese $n$-folds passing through ${n+d \choose d}+n+1$ general points.
\end{problem}
These numbers will be denoted $\nu_{d,n}$.  \Cref{problem:main} has seen virtually no advancement beyond the case of rational normal curves.  Arthur Coble, about a century ago, used his new theory of association to find that $9$ general points in $\P^5$ determine precisely $4$ $2$-uple Veronese surfaces \cite[~Theorem 19]{CobleAssociatedSO}.  The configuration of these four surfaces is as special as Coble's argument showing $\nu_{2,2}=4$. He discovered, using what is now called Goppa's lemma, that a unique elliptic normal sextic curve $E \subset \P^{5}$ interpolates through the $9$ points. This implies that the $2$-uple Veronese surfaces containing all $9$ points must entirely contain $E$, an exceptional circumstance. This means that each surface corresponds to choosing a square root of the degree $6$ line bundle $\O_{E}(1)$.  And so, Coble established a correspondence 
\begin{align}
\label{Coblecorrespondence}
\left\{ \begin{tabular}{c}
    \emph{$2$-uple Veronese surfaces} \\
     \emph{containing the 9 points} 
\end{tabular} \right\} \longleftrightarrow \left\{ \begin{tabular}{c}
    \emph{Line bundles $L$ on $E$} \\
     \emph{satisfying $L^{2} \simeq \O_{E}(1)$} 
\end{tabular} \right\},\end{align}  from which he  deduced $\nu_{2,2}=4$. 
There is currently no explanation for the number $4$ without the curve $E$. Specifically, no approach parallel to those available for the $n=1$ case is known.

 Our work in this paper solves the next case of \Cref{problem:main}. 

\begin{theorem}
\label{theorem:main}  Thirteen general points in $\P^{9}$ determine $4246$ $3$-uple Veronese surfaces, i.e. $\nu_{3,2} = 4246$.
\end{theorem}

A caricature of the proof best serves to explain the contents of the paper.  The first step is to use a correspondence like \eqref{Coblecorrespondence}, though more intricate. It is the content of \Cref{theorem:translation} in \S \ref{sec:translation}.  This correspondence trades the counting of $3$-uple Veronese surfaces for the counting of certain triples of points in the plane called \emph{singular triads} $T \in \Hilb_{3}\P^{2}$.  It first appeared in the work of the second author and A. Landesman \cite{landesman2019interpolation} on interpolation \footnote{It is quite plausible that Coble, the foremost expert on relationships between  association and Cremona transformations, knew about the correspondence in \Cref{theorem:translation}.  One hypothesis as to why he didn't write about this particular correspondence might be that the resulting problem of counting singular triads posed too many complications given the technology available at the time.}.  We give a thorough account to keep things self-contained.  

In order to count singular triads we need access to a vector bundle which to a point $T \in \Hilb_{3}\P^{2}$ assigns the vector space $H^{0}(\P^{2},\I^{2}_{T}(5))$ of quintic forms singular at $T$.  Unfortunately this vector bundle doesn't exist because of a familiar failure of flatness: the scheme obtained by squaring the ideal of a length $3$ scheme is not always a length $9$ scheme.  And so the second step is to deal with this non-vector-bundle. We swap out the Hilbert scheme for a birational modification we call the space of \emph{complete triangles} $\CT$. The name is chosen because of many similarities it shares with the space of complete conics.  We study the geometry of $\CT$ in \S \ref{sec:completetriangles}.  While it can be shown that $\CT$ is the quotient of the space of \emph{ordered} triangles studied by S. Keel in \cite{keel1993functorial} by the natural symmetric group action, our construction of $\CT$ is of independent interest as it does not require first ordering triangles. The vector bundle we sought in the previous paragraph exists over $\CT$, and we gain an enumerative understanding of it using Atiyah-Bott localization.  At least at first glance, the set of singular triads is the degeneracy scheme in $\CT$ of a vector bundle map involving our newfound bundle. Porteous' formula, implemented using \texttt{sage}, then suggests that the number we seek is $57728$.  

All is not well, however, because there is still a gnarly excess in the Porteous setup.  Our third step is to circumvent this new complication by further linearizing the problem,  switching to a $26$-dimensional Grassmannian bundle over $\CT$ which we call the space of \emph{singular quintic pencils} $\SQP$.  Only in $\SQP$ do we finally find an excess-free vantage point.  \emph{Proving} the lack of excess is painful, requiring a combination of dimension counting and limit linear series arguments.  This  verification is the subject of \S \ref{sec:pencilspace}, and is the content behind \Cref{theorem:Domintegral} which expresses $\nu_{3,2}$ as an integral: 
\begin{align}
\label{integral}
\nu_{3,2} = \int_{\SQP} \left[\Dom(p) \right]^{13}.
\end{align}  The details are not so important right now, but $\Dom(p)$ is a relevant codimension $2$ subvariety of $\SQP$.   In \S \ref{sec:finishing} we compute the integral \eqref{integral} using Atiyah-Bott localization, performed with the help of \texttt{sage}. Finally, in \S \ref{sec:questions} we discuss some of the many questions emerging from our investigation. We've included the sage code we used for the calculations in \S \ref{sec:sagecode}.

\subsection{Relation to other work} Apart from the obvious connection to Coble's work, the present paper is related to some other work which deserves mentioning.  Despite much of the progress on curve counting, there aren't many examples of counts of higher dimensional varieties.  The closest work in this sense is due to Coskun in \cite{coskun2006enumerative, coskun2006degenerations}.  While our construction of the space of (unordered) complete triangles is novel, it has an {\sl ordered} antecedent in the work of Collino and Fulton in \cite{collino1989intersection} and in the work of Keel in \cite{keel1993functorial}. 

\subsection{Notation and conventions}

Our ground field $\k$ is algebraically closed and of characteristic $0$.  All schemes considered in the paper are separated and of finite-type over $\k$.  If $A$ is a $\k$-vector space, and if $X$ is a scheme then $\underline{A}$ will denote the constant vector bundle on $X$ whose fibers are $A$. If $Z$ is a closed subscheme of a scheme $X$ then $\I_{Z}$ will denote its ideal sheaf.  

If $V$ is a vector bundle, then $\P V$ denotes its projectivization which represents lines in $V$.  In particular, $H^{0}(\P V, \O_{\P V}(1)) = V^{*}$ canonically.  Similarly $\Gr(k,V)$ denotes the Grassmannians representing $k$-dimensional subspaces.

\section{The correspondence}
\label{sec:translation}

The fundamental trick for computing $\nu_{3,2}$ is to switch to the task of counting corresponding {\sl triples} of
non-collinear points $\{a,b,c\} \subset \P^2$ satisfying certain conditions
relative to 13 prescribed, general points $\Gamma_{13} \subset \P^2$.  It is the content of \Cref{theorem:translation} in this section.  Whether
 similar useful correspondences are available for determining other $\nu_{d,n}$'s remains an intriguing question.  
 
 The correspondence critically uses Coble's theory of association, and we begin by
reviewing this theory following the incisive account in \cite{EISENBUD2000127}.  

\subsection{Association}
\label{sec:association}

We let $R$ be a Gorenstein $0$-dimensional $\k$-algebra of length $d$,
and we let $\Gamma = \Spec R$, and write \[\pi: \Gamma \to \Spec \k\] for
the structural morphism.  The Gorenstein condition says that the dualizing sheaf $\omega_{\pi}$,
associated to the $R$-module
\[ \Hom_{\k}(R,\k),\] is invertible and in fact generated by the trace
functional. The evaluation map
\[ev: \pi_{*}\omega_{\pi} \to \k\] sends a functional 
$f \in \Hom_{\k}(R,\k)$ to $f(1)$.

Given a line bundle $L$ on $\Gamma$ (equivalently a rank $1$
free $R$-module), one obtains a natural pairing
\[\langle, \rangle: \pi_{*}L \times 
  \pi_{*}\left( L^{-1} \otimes_{\O_{\Gamma}}
    \omega_{\pi}\right) \to \k \] which is a perfect pairing between
two $d$-dimensional $\k$-vector spaces, thanks to the Gorenstein
property.  So, if
$V \subset H^{0}(L) = \pi_{*}L$ is any
$r+1$-dimensional vector subspace,  then we can define its {\sl
  associated subspace}
\[V^{\perp} \subset H^{0}(L^{-1}\otimes \omega_{\pi})\] to be $V$'s orthogonal complement with respect to $\langle, \rangle$.

%%%We will need the following families-version in a transversality argument later.
\begin{remark}
The passage from $V$ to $V^{\perp}$ can also be done
in a relative setting: If $\pi : \mathcal{G} \to S$ is a finite, degree $d$
Gorenstein morphism of schemes, and if $\mathcal{L}$ is a line bundle
on $\mathcal{G}$, then a rank $r+1$ sub-bundle
\[\mathcal{V} \subset \pi_{*}\mathcal{L}\] (with locally free
quotient) has an associated rank $d-r-1$
sub-bundle
\[\mathcal{V}^{\perp} \subset \pi_{*} \left(\mathcal{L}^{-1} \otimes
    \omega_{\mathcal{G}/S}\right)\] with locally free quotient.
\end{remark}
%%%Goppa's theorem

How does one identify $V^{\perp}$ in practice?
Goppa's theorem provides an answer 
in a common geometric situation:

\begin{lemma}[Goppa]
  \label{lemma:goppa} Let $C$ be smooth projective curve, $L$ a
  non-special line-bundle on $C$, and $\Gamma \subset C$ a finite
  subscheme such that the restriction map
  $H^{0}(C,L) \to H^{0}(\Gamma, L_{\Gamma})$ is injective.  Then the
  image of the vector space
  $H^{0}(C, \omega_{C}(\Gamma) \otimes L^{-1})$ in
  $H^{0}(\Gamma, \omega_{\Gamma} \otimes L^{-1})$ induced by the
  adjunction isomorphism
  $\omega_{C}(\Gamma)|_{\Gamma} \otimes L^{-1} \to \omega_{\Gamma} \otimes L^{-1}$ is the associated
  space to $H^{0}(C,L)$.
\end{lemma}

Armed with Goppa's theorem, we're now ready to switch problems.

\subsection{The correspondence}

Let $\mathcal{H}$ denote the variety parametrizing $3$-uple Veronese surfaces in $\P^{9}$, and let $$\mathcal{X} \subset \mathcal{H} \times (\P^{9})^{13}$$ denote the irreducible, closed subvariety parametrizing tuples $([V],q_1,\dots, q_{13})$ satisfying $q_{i} \in V$ for all $i$.  We let $\mathcal{X}^{\circ} \subset \mathcal{X}$ denote the open subset  parametrizing tuples $$\left([V], q_1, \dots, q_{13}\right)$$ which satisfy the following conditions: 

\begin{enumerate}
    \item The points $q_{i}$ should be distinct;
    \item When we think of $V$ as a projective plane $\P^{2}$, the points $q_{i}$ should define a pencil of plane quartic curves whose base scheme consists of $\{q_1, \dots, q_{13}\}$ together with three distinct non-collinear points $R \subset \P^{2}$. Furthermore, the triangle spanned by $R$  should not contain any of the points $q_{i}$. 
\end{enumerate}
We let $$\pi: \mathcal{X} \to (\P^{9})^{13}$$ denote the map sending $([V],q_1, \dots, q_{13})$ to $(q_{1}, \dots, q_{13})$, and we note that $$\nu_{3,2} = \# \pi^{-1}(\{(q_1, \dots, q_{13})\})$$ for general choices of points $q_{i}$. 

Next, we let $\Hilb_{3}\P^{2}$ denote the Hilbert scheme parametrizing length $3$ subschemes of $\P^{2}$, and we define  $$\mathcal{Y} \subset \Hilb_{3}\P^{2} \times (\P^{2})^{13}$$ to be the locally closed subvariety parametrizing tuples $\left([T], p_1, \dots, p_{13}\right)$ satisfying the following conditions: 

\begin{enumerate}
    \item The length $3$ subscheme $T \subset \P^{3}$ is reduced, and is not contained in any line.
    \item The triangle spanned by $T$ does not contain any of the points $p_{i}$.
    \item The points $p_{i}$ are distinct.
    \item There exist two reduced, irreducible degree $5$ curves containing all points $p_{i}$ and singular at the three points of $T$.
\end{enumerate}

If $\left([T], p_{1}, \dots, p_{13} \right)$ is an element of $\mathcal{Y}$, we say $T$ \emph{is a {\bf singular triad} for $p_{1}, \dots, p_{13}$.}
We write $$\eta: \mathcal{Y} \to (\P^{2})^{13}$$ for the map sending $([T],p_{1}, \dots, p_{13})$ to $(p_{1}, \dots, p_{13})$.  Observe that the groups $\PGL(10)$ and $\PGL(3)$ act on $\P^{9}$ and $\P^{2}$, respectively, and furthermore induce natural actions on $\mathcal{X}, \mathcal{X}^{\circ}$ and $\mathcal{Y}$. 

\begin{theorem}
\label{theorem:translation}
Let $(q_{1}, \dots, q_{13}) \in (\P^{9})^{13}$ be a general tuple with associated tuple $(p_{1}, \dots, p_{13}) \in (\P^{2})^{13}.$  There exists a bijective correspondence 
\begin{align*}
    \left\{ \begin{tabular}{c}
         $3$-uple Veronese\\ surfaces $V \subset \P^{9}$ \\
         containing $q_{1}, \dots , q_{13}$
    \end{tabular}\right\} \longleftrightarrow \left\{ \begin{tabular}{c}
         Singular triads \\ $T \subset \P^{2}$ \\
         for $p_{1}, \dots, p_{13}$
    \end{tabular} \right\}.
\end{align*}
\end{theorem}

\begin{proof}
    $A$ and $B$ will denote the sets $\pi^{-1}(\{(q_{1}, \dots, q_{13})\})$ and $\eta^{-1}(\{(p_{1}, \dots, p_{13})\})$, respectively -- the objective is to show $A$ and $B$ are in bijection.  A simple dimension count, which we omit, shows that $A$ and $B$ are finite sets.  As $\mathcal{X}$ and $(\P^{9})^{13}$ are both  irreducible and $117$-dimensional, and because $(q_{1}, \dots, q_{13}) \in (\P^{9})^{13}$ is general, it follows that $A \subset \mathcal{X}^{\circ}$. 
    
    First we describe a function $$\Phi: A \to B.$$ Choose any $([V],q_{1}, \dots, q_{13}) \in A$ to begin with. When we interpret $V$ as a projective plane, the thirteen points $q_{i} \in V$ determine a general pencil of quartic plane curves $C_{t} \subset V$, $t \in \P^{1}$. Let $R \subset V$ denote the three points residual to $\{q_{1}, \dots, q_{13}\}$ in the base locus of the pencil $C_{t}$.  Since $A \subset \mathcal{X}^{\circ}$, when we view $V$ as a plane, the triangle in $V$ spanned by $R$ does not contain any of the points $q_{i}$, and the general member of $C_{t}$ is a smooth quartic curve.  

    Now let $$\mu: V \dashrightarrow \P^{2}$$ denote the quadratic Cremona transformation with indeterminacy set $R$, well-defined up to post-composition with elements of $\PGL(3)$. By applying Goppa's theorem to the divisor $q_{1} + \dots + q_{13}$ on any smooth member of the pencil $C_{t}$, we conclude that the tuple $(\mu(q_{1}), \dots, \mu(q_{13}))$ is associated to $(q_{1}, \dots, q_{13}).$  Let $R' \subset \P^{2}$ denote the three points which are the images of the three lines contracted under $\mu$.  Since $(p_{1}, \dots, p_{13})$ is associated to $(q_{1}, \dots, q_{13})$ by assumption, and since the latter tuple is general, it follows that there is a unique element  $g \in \PGL(3)$ which takes $(\mu(q_{1}), \dots, \mu(q_{13}))$ to $(p_{1}, \dots, p_{13})$. 
    
    Set $T := g (R')$. We will verify the membership $$([T],p_{1}, \dots, p_{13}) \in B,$$ and then declare $$\Phi([V],q_{1}, \dots, q_{13}) := ([T],p_{1}, \dots, p_{13}).$$  To that end, we must show $([T],p_{1}, \dots, p_{13})$ satisfies the requirements for membership in  $\mathcal{Y}.$ First, $T$ is a reduced, non-collinear set of three points because the same was true for $R$. Second, if some point  $p_{i}$ was contained in the triangle spanned by $T$, then by applying an appropriate reverse Cremona transformation, it would follow that the corresponding point $q_{i}$ was contained in the triangle spanned by $R$, contrary to assumption. Third, the points $p_{i}$ are distinct because no two $q_{j}$'s are contained in a line spanned by two of the points of $R$. Fourth and finally, by considering the curves $\mu(C_{t})$ for $t \in \P^{1}$ general, we find at least two irreducible quintic curves singular at $T$ and passing through $(p_{1}, \dots, p_{13})$. 

    The function $\Phi: A \to B$ having been defined, we now define a mapping $\Psi: B \to A$ which is readily seen to be its inverse.  To start, choose $([T],p_{1}, \dots, p_{13}) \in B$. Let $\widetilde{\P}^{2} \to \P^{2}$ denote the blow-up of $\P^{2}$ at the three points of $T$.  On $\widetilde{\P}^{2}$, let $L$ and $E$ denote the divisor classes of a general line in $\P^{2}$ (pulled back to the blow-up) and the sum of the three exceptional curves, respectively, and let $Q,Q' \in |5L-2E|$ be the strict transforms of two of the quintic curves mentioned in the membership requirements for $\mathcal{Y}$.  Observe that $Q\cap Q' = \{p_{1}, \dots, p_{13}\}$. By applying Goppa's theorem to the divisor $p_{1}+ \dots + p_{13}$ on $Q$, it follows that the map $$\gamma: \widetilde{\P}^{2} \to \P^{9}$$ given by the complete linear series $|6L-3E|$ is such that the tuple $(\gamma(p_{1}), \dots, \gamma(p_{13})) \in (\P^{9})^{13}$ is associated to $(p_{1}, \dots, p_{13}) \in (\P^{2})^{13}$.  Much as before, there is a unique element $h \in \PGL(10)$ which sends $(\gamma(p_{1}), \dots, \gamma(p_{13}))$ to $(q_{1}, \dots, q_{13})$. The map $$h \circ \gamma: \widetilde{\P}^{2} \to \P^{9}$$ has as image a $3$-uple Veronese surface $V$ containing the points $q_{i}$ for all $i$. Define $\Psi$ by $\Psi([T],p_{1},\dots, p_{13}) := ([V],q_{1}, \dots, q_{13}).$   $\Phi$ and $\Psi$ are inverses, proving the theorem.
\end{proof}

\subsection{Returning to the original objective}
\label{subsec:returning}
The stated objective in this paper is to determine $\#\pi^{-1}(\{(q_{1}, \dots, q_{13})\})$ for general points $q_{i} \in \P^{9}$.  Using \Cref{theorem:translation}, we instead will try to compute $$\#\eta^{-1}(\{p_{1}, \dots, p_{13}\})$$ for general points $p_{i} \in \P^{2}$.  We will still face several obstacles.

The first obstacle arises when we want to refer to the rank $9$ ``vector bundle" $E$ on $\Hilb_{3}\P^{2}$ whose fiber over a point $[T]$ is the vector space $$H^{0}\left(\P^{2}, \left(\O_{\P^{2}}/ \I_{T}^{2}\right)\otimes \O_{\P^{2}}(5)\right).$$  With such a bundle we can apply the Porteous formula to access the locus where the natural morphism $$\underline{H^{0}\left(\P^{2},\I_{\{p_{1}, \dots, p_{13}\}}(5)\right)} \to E$$ has at least a $2$-dimensional kernel -- the key membership condition defining the set $\eta^{-1}(\{(p_{1}, \dots, p_{13})\})$. Unfortunately, $E$ is not actually a vector bundle because its rank jumps from $9$ to $10$ over the locus parametrizing fat schemes.  And so our attention turns to replacing $\Hilb_{3}\P^{2}$ with a better parameter space, the space of {\sl complete triangles} $\CT$, which resolves this jumping wrinkle.  We do all this and more in the next section, which is dedicated to the geometry and construction of $\CT$.

\section{The space of complete triangles}
\label{sec:completetriangles}

Let $\Hilb_{3}\P^{2}$ and $\Hilb_{3}\check{\P}^{2}$ denote the Hilbert
schemes of $0$-dimensional, length $3$ subschemes of a projective
plane and its dual plane, respectively.  Upon choosing coordinates $[X:Y:Z]$ for $\P^{2}$ and $[\check{X}:\check{Y}:\check{Z}]$ for $\check{\P}^{2}$, the group
$\PGL(3)$ acts on both Hilbert schemes naturally.
Under this action, $\Hilb_{3}\P^{2}$ decomposes into $7$ orbits -- (A) three non-collinear points, (B) three collinear points, (C) a length two point and a reduced point, noncollinear, (D) a length two point and a collinear point, (E) A length 3 non-reduced subscheme of a conic, (F) a length 3 non-reduced subscheme of a line, (G) a fat point, given by the square of a maximal ideal.  Of these, (F) and (G) are closed.   Our objective in this section is to construct and analyze a $\PGL(3)$-equivariant modification of $\Hilb_{3}\P^{2}$ which we call the space of {\sl complete triangles}, and which we denote by $\CT$.

\subsection{Nets of conics, Jacobian spaces, and constructing $\CT$}

\begin{definition}
  \label{definition:NetZ} 
  \begin{enumerate}
    \item Let $T \in \Hilb_{3}\P^{2}$ be a length three scheme. $T$'s {\bf net
    of conics} is the vector space
  $\Lambda_{T} := H^{0}(\I_{T}(2)) \subset H^{0}(\O_{\P^{2}}(2)).$
  \item The {\bf Jacobian matrix} of three homogeneous quadratic forms
  $Q_{1},Q_{2},Q_{3}$ in the variables $X,Y,Z$ is:
  \begin{eqnarray}\label{jacmatrix}
    \begin{bmatrix}
      \frac{\partial Q_{1}}{\partial X} & \frac{\partial Q_{1}}{\partial Y} & \frac{\partial Q_{1}}{\partial Z} \\
      \frac{\partial Q_{2}}{\partial X} & \frac{\partial Q_{2}}{\partial Y} & \frac{\partial Q_{2}}{\partial Z} \\
      \frac{\partial Q_{3}}{\partial X} & \frac{\partial Q_{3}}{\partial Y} & \frac{\partial Q_{3}}{\partial Z} \\
    \end{bmatrix}
  \end{eqnarray}
  \item The {\bf Jacobian space} of $V \subset H^{0}(\P^{2}, \O(2))$ is the vector space spanned by all $2 \times 2$ minors of the Jacobian matrix of any basis $\langle Q_{1}, Q_{2}, Q_{3} \rangle$ of $V$.
  \item If $V \subset H^{0}(\P^{2}, \O_{\P^{2}}(2))$ is $3$-dimensional, we
  let $V^{*} \subset H^{0}(\check{\P}^{2}, \O_{\check{\P}^{2}}(2))$ denote its apolar
  $3$-dimensional space.
  \item A scheme $T \subset \P^{2}$ is {\bf fat} if it is isomorphic to
  $\spec k[x,y]/(x^{2},xy,y^{2})$. A scheme $T \subset \P^{2}$ is {\bf
    thin} if it is contained in a line. We define $\Fat, \Thin \subset \Hilb_{3}\P^{2}$ to be the
  closed loci of fat and thin schemes, respectively.
  \item A subscheme $T \subset \P^{2}$ (or $\check{\P}^{2}$) consisting of three distinct non-collinear points is called an {\bf honest triangle}
  \end{enumerate}
\end{definition}

\begin{remark}
\label{remark:alwaysthreedimensional}
\begin{enumerate}
  \item $\Lambda_{T}$ is a $3$-dimensional vector space, no matter the
  scheme $T$, as is easily checked for each of the $7$ $\PGL(3)$
  orbits separately.
  \item Apolarity is the natural pairing between $H^{0}(\P^{2}, \O_{\P^{2}}(2))$ and $H^{0}(\check{\P}^{2}, \O_{\check{\P}^{2}}(2))$, where the latter space is viewed as homogeneous second order differential operators on the former space. So, for instance, the pairing outputs $2$ for the pair $X^{2}, \check{X}^{2}$.
\end{enumerate}
\end{remark}

\begin{proposition}
  \label{proposition:JacobianNet}
  Let $T \subset \P^{2}$ be any length three scheme, $\Lambda_{T}$ its
  net of conics. Then the space $\jac(\Lambda_{T}^{*})$ is 
  $3$-dimensional.

\end{proposition}

\begin{proof}
  By projective equivariance of the assignment
  $T \mapsto \jac(\Lambda_{T}^{*})$, it suffices to check the
  proposition on seven representatives of the $\PGL(3)$-orbits, which
  we omit.
\end{proof}

\begin{definition}
  \label{definition:dagger}
  Let $T \subset \P^{2}$ be a length $3$ scheme. We set
  $\Lambda_{T}^{\dagger} := \jac(\Lambda_{T}^{*})$.
\end{definition}

\subsection{Examples}
Let us give some calculations related to some of the things we've introduced. As a reminder, $[X:Y:Z]$ are homogeneous
coordinates in $\P^{2}$, and $[\check{X}:\check{Y}:\check{Z}]$ are dual
coordinates. Thus, the point $[\check{X}:\check{Y}:\check{Z}]$ represents the line in $\P^{2}$
defined by the equation $\check{X}X+\check{Y}Y+\check{Z}Z = 0$. Brackets $\langle \rangle $
will denote ``$\k$-linear span''.

\begin{example}
  \label{example:coordinatepoints}
  Let $T \subset \P^{2}$ be the three coordinate points (orbit (A)), so that
  $\Lambda_{T} = \langle XY, YZ, XZ \rangle $. Then
  $\Lambda_{T}^{*} = \langle \check{X}^{2}, \check{Y}^{2}, \check{Z}^{2} \rangle $, whose
  Jacobian space is $\langle \check{X}\check{Y}, \check{Y}\check{Z}, \check{X}\check{Z} \rangle
  $.   Therefore,
  \begin{eqnarray*}
    \Lambda_{T}^{\dagger} = \langle \check{X}\check{Y}, \check{Y}\check{Z}, \check{X}\check{Z} \rangle, 
  \end{eqnarray*}
  which is the net of conics for the coordinate points in $\check{\P}^{2}$.
\end{example}

\begin{example}
    \label{example:curvillinearconic} Let $T$ be a length three non-reduced subscheme of a conic -- orbit (E).  $T$ is given by the vanishing scheme of the net $\Lambda_{T} = \langle XY, X^{2}, YZ \rangle$. Therefore, $\Lambda^{*}_{T} = \langle \check{Y}^{2}, \check{Z}^{2}, \check{X}\check{Z} \rangle$.  

    Computing the Jacobian space, we get $$\Lambda_{T}^{\dagger} = \langle \check{Y}\check{Z}, \check{Y}\check{X}, \check{Z}^{2} \rangle$$ which is the net of conics of a length three subscheme of a smooth conic in $\check{\P}^{2}$. 
\end{example}

\begin{example}
    \label{example:collinearZnet}
    Now suppose $T$ is a length three scheme contained in the line
    $L$ given by $Z=0$ (orbit (F)).  Then
    $\Lambda_{T} = \langle XZ, YZ, Z^{2} \rangle$.  Note that this net depends only on $L$, regardless of the particular
    scheme $T \subset L$.
  
    The dual space $\Lambda_{T}^{*}$ is
    $\langle \check{X}^{2},\check{X}\check{Y}, \check{Y}^{2} \rangle$, which is also easily checked to
    be its own Jacobian. Therefore,
    \begin{eqnarray*}
      \Lambda_{T}^{\dagger} = \langle \check{X}^{2}, \check{X}\check{Y}, \check{Y}^{2} \rangle.
    \end{eqnarray*}
    
  \end{example}

\begin{example}
  \label{example:fatpointZnet}
  Let $T$ be the fat point supported at $[0:0:1]$ (orbit (G)).  Its net of
  conics is $\Lambda_{T} = \langle X^{2}, XY, Y^{2}\rangle$. The dual
  space $\Lambda_{T}^{*}$ is $\langle \check{X}\check{Z}, \check{Y}\check{Z}, \check{Z}^{2} \rangle$. This
  latter space is its own Jacobian. Therefore
  \begin{eqnarray*}
    \Lambda_{T}^{\dagger} = \langle \check{X}\check{Z}, \check{Y}\check{Z}, \check{Z}^{2} \rangle.
  \end{eqnarray*}
  
\end{example}

\begin{proposition} Let $T \in \Hilb_{3}\P^{2}$ be arbitrary. Then $\Lambda_{T}^{\dagger}$ is the net of conics for some (not
    necessarily unique) $T^{*} \in \Hilb_{3}\check{\P}^{2}$.

\end{proposition}
\begin{proof}
  One checks this orbit by orbit -- we have done the most interesting examples above.  
\end{proof}

We can now state our main definition:

\begin{definition}\label{def:CT}
  The moduli space of {\bf complete triangles} is the closed subscheme
  $$\CT \subset \Hilb_{3}\P^{2} \times \Hilb_{3}\check{\P}^{2}$$ parametrizing
  pairs $(T, T^{*})$ satisfying
  \[\Lambda_{T}^{\dagger} = \Lambda_{T^{*}}.\]
\end{definition}

\begin{remark}
  The scheme structure on $\CT$ is induced by the condition
  $\Lambda_{T}^{\dagger} = \Lambda_{T^{*}}.$ Indeed, by
  \Cref{proposition:JacobianNet} and
  \Cref{remark:alwaysthreedimensional}, the assignments
  $(T, T^{*}) \mapsto \Lambda_{T}^{\dagger}$ and
  $(T, T^{*}) \mapsto \Lambda_{T^{*}}$ yield two vector subbundles of
  the trivial bundle $\underline{H^{0}(\check{\P}^{2}, \O_{\check{\P}^{2}}(2))}$
  over $\Hilb_{3}\P^{2} \times \Hilb_{3}\check{\P}^{2}$.  Requiring these two subbundles to be equal yields the scheme structure on
  $\CT$. $\CT$'s functor of points is inherent in this description.
\end{remark}

The diagonal action of $\PGL(3)$ on $\P^{2} \times \check{\P}^{2}$ induces an action on $\CT$  thanks to the projective equivariance of the assignment
\[\Lambda_{T} \mapsto \Lambda_{T}^{\dagger}.\]

From these calculations and projective equivariance of the
construction $\Lambda_{T} \mapsto \Lambda_{T}^{\dagger}$, we conclude:
\begin{proposition}
  \label{proposition:fatandtriple}
  Suppose $(T, T^{*}) \in \CT$. Then:
  \begin{enumerate} 
  \item $T$ is an honest triangle if and only
    if $T^{*}$ is an honest triangle. In this case, the points of
    $T^{*}$ correspond to the lines of the triangle spanned by pairs of points of $T$.

  \item If $T$ is a non-reduced scheme which is neither fat nor thin,
    then $T^{*}$ is unique, and vice versa.

  \item $T$ is a fat scheme if and only if $T^{*}$ is a thin scheme,
    and $T^{*}$ is contained in the line dual to the support of
    $T$. The same statement holds with $T$ and $T^{*}$ interchanged.

    \item $T^{*}$ is uniquely determined by $T$ if and only $T$ is not fat.
\end{enumerate}
\end{proposition}
\begin{proof}
  Follows from calculations similar to those in \Cref{example:coordinatepoints}, \Cref{example:fatpointZnet} and \Cref{example:collinearZnet} -- we leave the details to the reader.
\end{proof}

\begin{remark}
  From \Cref{example:collinearZnet}, notice that if $T$ is thin
  and $(T, T^{*}) \in \CT$, then $T^{*}$ is the fat scheme supported
  on the point dual to the line containing $T$.
\end{remark}

\begin{proposition}\label{prop:irreducible}
  The reduction $\CT_{\red}$ is $6$-dimensional and irreducible.
\end{proposition}

\begin{proof}
  From \Cref{proposition:fatandtriple}, it suffices to show: Given
  a pair $(T, T^{*}) \in \CT$ with $T$ fat, there exists an
  irreducible pointed curve $(B,0)$ and a map $f \from B \to \CT$ such
  that $f(0) = (T, T^{*})$ and for all $b \in B \setminus \{0\}$, $f(b)$
  is an honest triangle.  This is sufficient because the open locus of
  honest triangles is clearly $6$-dimensional and irreducible. Note
  that by symmetry, we need not consider the case where $T$ is thin.

  Since $T$ is fat, \Cref{proposition:fatandtriple} says $T^{*}$ is
  thin. In any case, since the open subset of $\Hilb_{3}(\check{\P}^{2})$
  parametrizing triples of three distinct, non-collinear points is
  Zariski dense, there exists a pointed curve $(B,0)$ and a map
  \begin{eqnarray*}
    f \from B \to \Hilb_{3}\check{\P}^{2}
  \end{eqnarray*}
  with $f(0) = T^{*}$, and $f(b)$ a triple of three non-collinear
  points, for all $b \neq 0$. For all points $b \in B$, the space
  $\Lambda_{f(b)}^{\dagger}$ determines a unique length three scheme
  $T_{b} \in \Hilb_{3}(\P^{2})$ such that
  $\Lambda_{f(b)}^{\dagger} = \Lambda_{T_{b}}$. Therefore, the map $f$
  lifts to a map $f \from B \to \CT$, and this lift has the desired
  properties, namely that $f(0) = (T, T^{*})$ and $f(b)$ is an honest triangle
  for $b \neq 0$.
\end{proof}

\subsection{Smoothness} Our next major objective is to show that $\CT$
is smooth. We do so by establishing smoothness at a particular point
$(T, T^{*}) \in \CT$, and then concluding by exploiting the $\PGL(3)$ action,  upper semi-continuity and
\Cref{prop:irreducible}.

\begin{proposition}\label{prop:smoothpoint} Let
  $(T, T^{*}) \in \CT$ be the complete triangle with $T$ given by the
  homogeneous ideal $(X^{2},XY,Y^{2})$, and $T^{*}$ given by the ideal
  $(\check{Z}, \check{X}^{3})$. The scheme $\CT$ is smooth at $(T, T^{*})$.
\end{proposition}

\begin{proof}
  We will calculate the space of first order deformations of
  $(T, T^{*}) \in \CT$, and demonstrate that it is $6$ dimensional. This is enough by \Cref{prop:irreducible}.

  Let us pass to affine coordinates; we let
  $x = X/Z, y = Y/Z, a = \check{X}/\check{Y}, c = \check{Z}/\check{Y}$. The general first order
  deformation of the ideal $I = (x^{2},xy,y^{2})$ is given by

\begin{eqnarray}\label{eqn:deformI}
  I_{\varepsilon} := (x^{2} + \varepsilon(\alpha_{1}x + \beta_{1}y), xy + \varepsilon(\alpha_{2}x + \beta_{2}y), y^{2}+ \varepsilon(\alpha_{3}x+\beta_{3}y)),
\end{eqnarray}
for free choices of constants $\alpha_{i}, \beta_{j} \in \k$, while the
general first order deformation of $J = (a^{3}, c)$ is
\begin{eqnarray}\label{eqn:deformJ}
  J_{\varepsilon} = (a^{3} + \varepsilon(\gamma_{1}a^{2}+\gamma_{2}a+\gamma_{3}), c + \varepsilon(\delta_{1}a^{2}+\delta_{2}a + \delta_{3}))
\end{eqnarray}
where $\gamma_{i}$ and $\delta_{i}$ vary freely in $\k$.

From \eqref{eqn:deformI}, the corresponding first order deformation of
the induced net of conics
$\Lambda_{T} = \langle X^{2}, XY, Y^{2} \rangle$ is given by
\begin{eqnarray}\label{eqn:deformnet}
  \Lambda_{T}(\varepsilon) := \langle X^{2} + \varepsilon(\alpha_{1}XZ + \beta_{1}YZ), XY + \varepsilon(\alpha_{2}XZ + \beta_{2}YZ), Y^{2}+ \varepsilon(\alpha_{3}XZ+\beta_{3}YZ) \rangle.
\end{eqnarray}

Next, we must identify the first order deformation of the net of
conics $\Lambda_{T^{*}} = \langle \check{X}\check{Z}, \check{Y}\check{Z}, \check{Z}^{2} \rangle$ induced by
the deformation \eqref{eqn:deformJ}.  First, note that
$c^{2} \in J_{\varepsilon}$. Therefore, we must determine how the
conics $\check{X}\check{Z}$ and $\check{Y}\check{Z}$ must be deformed.  By homogenizing the element
$c + \varepsilon(\delta_{1}a^{2}+\delta_{2}a + \delta_{3}) \in
J_{\varepsilon}$, we get the deformation:
\begin{eqnarray*}
  \check{Y}\check{Z} + \varepsilon(\delta_{1}\check{X}^{2}+\delta_{2}\check{X}\check{Y} + \delta_{3}\check{Y}^{2}).
\end{eqnarray*}
Finally, by multiplying
$c + \varepsilon(\delta_{1}a^{2}+\delta_{2}a + \delta_{3})$ by $a$,
subtracting $\varepsilon \delta_{1} a^{3}$, then homogenizing we
deduce the following deformation:
\begin{eqnarray*}
  \check{X}\check{Z} + \varepsilon(\delta_{2}\check{X}^{2} + \delta_{3}\check{X}\check{Y})
\end{eqnarray*} Altogether, the corresponding first order deformation of the net $\Lambda_{T^{*}}$ is given by  
\begin{eqnarray}
  \label{deformLambdaprime}
  \Lambda_{T^{*}}(\varepsilon) := \langle \check{X}\check{Z} + \varepsilon(\delta_{2}\check{X}^{2} + \delta_{3}\check{X}\check{Y}), \check{Y}\check{Z} + \varepsilon(\delta_{1}\check{X}^{2}+\delta_{2}\check{X}\check{Y} + \delta_{3}\check{Y}^{2}), \check{Z}^{2} \rangle 
\end{eqnarray}

Now, we must impose the condition that the pair of deformations
$(I_{\varepsilon}, J_{\varepsilon})$ belong in $\CT$. By definition,
this is the condition
\begin{eqnarray}\label{eqn:ConditionCT}
  \Lambda_{T}(\varepsilon)^{\dagger} = \Lambda_{T^{*}}
\end{eqnarray}

The left hand side of \eqref{eqn:ConditionCT} is simple to calculate.  First, the deformation $\Lambda_{T}(\varepsilon)^{*}$ is given by: 

\begin{eqnarray*}
  \Lambda_{T}(\varepsilon)^{*} = \langle \check{X}\check{Z} + \varepsilon(-\frac{\alpha_{1}}{2} \check{X}^{2} - \alpha_{2}\check{X}\check{Y} - \frac{\alpha_{3}}{2} \check{Y}^{2}), \check{Y}\check{Z} + \varepsilon(-\frac{\beta_{1}}{2} \check{X}^{2}-\beta_{2}\check{X}\check{Y} - \frac{\beta_{3}}{2} \check{Y}^{2}), \check{Z}^{2} \rangle
\end{eqnarray*}
and so applying the Jacobian we get 

\begin{eqnarray*}
    \Lambda_{T}(\varepsilon)^{\dagger} = \langle \check{X}\check{Z} + \varepsilon(\beta_{2}\check{X}^{2} + (\alpha_{2}-\beta_{3})\check{X}\check{Y} + \alpha_{3}\check{Y}^{2}), \check{Y}\check{Z} + \varepsilon(\beta_{1}\check{X}^{2}+(\beta_{2}-\alpha_{1})\check{X}\check{Y} - \alpha_{2}\check{Y}^{2}), \check{Z}^{2} \rangle
\end{eqnarray*}
Equating term by term with \eqref{deformLambdaprime} we get the following system of equations:

\begin{eqnarray*}\label{SystemOfEquations}
  \delta_{2} =\beta_{2}\\
  \delta_{3} = \alpha_{2}-\beta_{3}\\
  0 = \alpha_{3}\\
  \delta_{1} = \beta_{1}\\
   \delta_{2} = \beta_{2}-\alpha_{1}\\
  \delta_{3} = -\alpha_{2}\\
\end{eqnarray*}

We note that this system of equations clearly cuts out a
$6$-dimensional space of solutions in the vector space
$\k\langle \alpha_{1}, \beta_{1}, \alpha_{2}, \beta_{2}, \alpha_{3},
\beta_{3}, \gamma_{1}, \gamma_{2}, \gamma_{3}, \delta_{1}, \delta_{2},
\delta_{3} \rangle$ -- indeed, we may freely choose
$\gamma_{1}, \gamma_{2}, \gamma_{3}, \delta_{1}, \delta_{2},
\delta_{3} $, after which the $\alpha$ and $\beta$ variables are
uniquely determined, concluding the proof.
\end{proof}

\begin{remark}
\label{remark:delta1transversality} In the setting of the proof of \Cref{prop:smoothpoint}, we note that the extra single condition $\delta_{1} = 0$ is equivalent to the condition that the $1$'st order deformation of the ideal $J_{\varepsilon}$ continues to contain a linear element.  But then the allowable parameters $\alpha_{1}, \beta_{1}, \alpha_{2}, \beta_{2}, \alpha_{3}, \beta_{3}$ determine a deformation of the ideal $I$ of the form $((x-\varepsilon \alpha)^{2}, (x - \varepsilon \alpha)(y - \varepsilon \beta), (y - \varepsilon \beta)^{2})$.  This is evidently the tangent space to $\Fat$ at the point $T$.

Furthermore, a deformation of $J$ with $\delta_{1} \neq 0$ induces a non-zero tangent vector of $T$ which is a nonzero normal vector to $\Fat$ at $T$.
\end{remark}

\begin{theorem}
  \label{thm:smoothCT}
  The scheme $\CT$ is smooth.
\end{theorem}
\begin{proof}
  We use upper semi-continuity, \Cref{prop:smoothpoint}, and
  \Cref{prop:irreducible}. The group $\PGL(3)$ acts on the scheme
  $\CT$.  Under this action, it is easy to check that either the point
  $(T, T^{*})$ from \Cref{prop:smoothpoint} or its symmetric flip
  (where we interchange $X,Y,Z$ with $\check{X},\check{Y},\check{Z}$ everywhere) is contained
  in the closure of every orbit.  We conclude that the tangent space
  is $6$-dimensional at every point in $\CT$ by semi-continuity, and
  conclude by \Cref{prop:irreducible}.
\end{proof}

\begin{corollary}\label{cor:closureofgraph}
  The space of complete triangles $\CT$ is the 
  closure (with reduced, induced scheme structure) of the graph of the
  triangle map
  \[
    \tau \from \Hilb_{3}\P^{2} \dashrightarrow \Hilb_{3}\check{\P}^{2}
  \]
  which assigns to a set of three distinct, non-collinear points
  $T \subset \P^{2}$, the set of lines spanned by pairs in $T$.
\end{corollary}

\begin{proof}
  This follows from part $(1)$ of \Cref{proposition:fatandtriple},
  \Cref{prop:irreducible}, and \Cref{thm:smoothCT}.  (The latter
  is needed to show that $\CT$ is reduced.)
\end{proof}

Let $\phi_{1} \from \CT \to \Hilb_{3}\P^{2}$ and
$\phi_{2} \from \CT \to \Hilb_{3}\check{\P}^{2}$ denote the forgetful maps.
\begin{corollary}
  \label{cor:blow upfat}
  The space $\CT$ is isomorphic to the blow up
  $\Bl_{\Fat}\Hilb_{3}\P^{2}$ and to the blow up
  $\Bl_{\Fat}\Hilb_{3}\check{\P}^{2}$. Under these isomorphisms, the
  forgetful maps $\phi_{i}$ are the respective blow down maps.
\end{corollary}
\begin{proof}
  We prove that $\CT \simeq \Bl_{\Fat}\Hilb_{3}\P^{2}$ -- the argument
  for $\Bl_{\Fat}\Hilb_{3}\check{\P}^{2}$ is the same.

  The projection $\phi_{1} \from \CT \to \Hilb_{3}\P^{2}$ is such that
  $\phi^{-1}(\Fat)$ has codimension $1$ in $\CT$, from the set-theoretic statement (3) in \Cref{proposition:fatandtriple}.  This preimage is also smooth, thanks to the observation in \Cref{remark:delta1transversality}, combining the $\PGL(3)$ action with upper semi-continuity.  
  
  The universal property of blow ups gives us an induced morphism
  $\iota \from \CT \to \Bl_{\Fat}\Hilb_{3}\P^{2}$ factoring
  $\phi_{1}$.  We conclude by applying  Zariski's Main Theorem to the
  morphism $\iota$.
\end{proof}

\subsection{Resolving powers} % (fold)
\label{sub:nthpowermaps}
Let $n \geq 2$ be an integer, and let
\[\mu_{n} \from \Hilb_{3}\P^{2} \dashrightarrow \Hilb_{3 {n+1 \choose
      2}}\P^{2}\] denote the rational map defined by sending an ideal
sheaf $\I$ to the ideal sheaf $\I^{n}$. The maps $\mu_{n}$ fail to be
defined along $\Fat \subset \Hilb_{3}\P^{2}$, since the colength of
$\I^{n}$ is ${2n+1 \choose 2}$ rather than $3 {n+1 \choose 2}$ when
$\I = \mathfrak{m}_{p}^{2}$. Here $\mathfrak{m}_{p} \subset \O_{\P^{2}}$
denotes the ideal sheaf of the point $p \in \P^{2}$. The excess in
colength is ${n \choose 2}$.  Conveniently, $\mu_{n}$ extends to a
morphism on $\CT$ for all $n$, as we now explain. Let
$$\gamma \from \Hilb_{3}\check{\P}^{2} \to \P H^{0}(\P^{2}, \O(3))$$ denote
the composition of the Hilbert-Chow morphism
$\Hilb_{3}\check{\P}^{2} \to \Sym^{3}\check{\P}^{2}$ with the natural
``multiplication'' morphism
\begin{eqnarray*}
  m \from \Sym^{3}\check{\P}^{2} \to \P H^{0}(\P^{2}, \O(3))
 \end{eqnarray*}
 induced by the $S_{3}$-equivariant ``multiplication of linear forms"
 map
 $$\check{\P}^{2} \times \check{\P}^{2} \times \check{\P}^{2} \to \P H^{0}(\P^{2}, \O(3)).$$
 We view points of $\P H^{0}(\P^{2}, \O(3))$ as cubic curves in
 $\P^{2}$.

 Therefore we see that every complete triangle $(T, T^{*}) \in \CT$ comes
 naturally equipped with a cubic curve $\gamma(T^{*}) \subset \P^{2}$
 which is a union of lines, possibly with multiplicities. If
 $(T, T^{*})$ is an honest triangle, then $\gamma(T^{*})$ is simply
 the triangle spanned by $T$ -- in particular, $\gamma(T^{*})$ gives a
 global section of $\I_{\Gamma}^{2}(3)$. If
 $(T, T^{*}) \in \CT$ is such that $T$ is a fat scheme,
 then $\gamma(T^{*})$ is an ``asterisk'' consisting of three lines
 passing through the support point of $T$.

\begin{theorem}
  \label{thm:resolventhpower}
  The rational map $\mu_{n}$ extends to a morphism on $\CT$ for all
  $n$.
\end{theorem}

\begin{proof} For simplicity of notation, we let $\gamma$ denote the
  ideal sheaf generated by the cubic $\gamma(T^{*})$. Consider the
  assignment
  \begin{eqnarray}
    \label{eqn:modifiedideal}
  (T, T^{*}) \mapsto \I_{T}^{n} + \gamma \I_{T}^{n-2} + \gamma^{2}\I_{T}^{n-4} + ...
\end{eqnarray}
where the sum ends at $\gamma^{n/2}$ or $\gamma^{(n-1)/2}\I_{T}$
depending on the parity of $n$. We leave it to the reader to check
that the ideal on the right side of \Cref{eqn:modifiedideal} has
the correct colength $3{n+1 \choose 2}$, even if $T$ is a fat
scheme. Hence, this assignment induces a morphism
$\CT \to \Hilb_{3 {n+1 \choose 2}}\P^{2}$.

Finally, if $(T, T^{*})$ is an honest triangle, then
$\gamma$ is already contained in $\I_{T}^{2}$ and therefore the
above assignment reduces to $\I_{T} \mapsto \I_{T}^{n}$,
 the $n$th power map, concluding the proof.
\end{proof}

\begin{corollary}
  \label{cor:AnySurface}
  Let $S$ be any smooth surface. Then the $n$th power map
  \[\mu_{n} \from \Hilb_{3}S \dashrightarrow \Hilb_{3 {n+1 \choose
      2}}S\] extends to  a regular map on $\Bl_{\Fat}\Hilb_{3}S$.
\end{corollary} 
\begin{proof}
  It suffices to prove the statement  \'etale locally on $S$.
  However, \'etale locally $S$ is isomorphic to $\P^{2}$.  The result
  follows from \Cref{cor:blow upfat} and
  \Cref{thm:resolventhpower}.
\end{proof}

Returning to the overarching narrative, in order to count the number of points in $$\eta^{-1}(\{p_{1}, \dots, p_{13}\}) \subset \Hilb_{3}\P^{2}$$ we will use the squaring map $\mu_{2}$ to
produce a map between two vector bundles on $\CT$.  An
appropriate degeneracy scheme of this map will include
$\eta^{-1}(\{p_{1}, \dots, p_{13}\})$ along with undesirable excess. We will remove the excess in \Cref{sec:pencilspace} -- for now, we content ourselves with identifying the relevant vector bundles.

\subsection{Some vector bundles on CT}
\label{sec:vectorbundle}

We simplify notation and let
\[\s : \CT \to \Hilb_{9}\P^{2} \] denote the squaring map
$\mu_{2}$ from the previous section.  We let
\[\mathcal{Z} \subset \CT \times \P^{2}\] denote the universal length
$9$ subscheme induced by $\s$, and write $p_1, p_2$ for the projections of
$\CT \times \P^{2}$ to first and second factors, respectively.  Finally, we denote by $\mathcal{V}_{n}$ the sheaf 
\begin{align}
\label{Vn}
p_{1*}\left(\mathcal{I}_{\mathcal{Z}} \otimes p_{2}^{*}\O(n) \right).
\end{align}
\begin{proposition}
  \label{proposition:vectorbundle}
  $\mathcal{V}_{n}$  is
  a vector bundle on $\CT$ of rank ${n+2 \choose 2} - 9$ for all $n \geq 5$.
\end{proposition}

\begin{proof}
  The claim being made is that every length $9$ scheme of the form  $\s(T,T^{*})$
  imposes independent conditions on degree $n \geq 5$ homogeneous
  forms. This can either be
  checked on a case-by-case basis, or can be checked for example on
  the special complete triangle $(T,T^{*})$ in
  \Cref{prop:smoothpoint}, as well as its dual. Then an appeal to semi-continuity gives
  the conclusion for any complete triangle $(T,T^{*})$ where $T$ is
  not contained in a line.  For those complete triangles $(T,T^{*})$
  with $T$ contained in a line, the claim can be checked directly.
\end{proof}

\begin{lemma}
    \label{lemma:generatedquintics} Let  $(T,T^{*}) \in \CT$ be any point such that $T \notin \Thin$. Then the sheaf $\I_{\s(T,T^{*})}(n)$ is generated by global sections for $n \geq 4$, i.e. the scheme defined by the common vanishing of all forms in $H^{0}(\P^{2},\I_{\s(T,T^{*})}(n))$ is precisely $\s(T,T^{*})$ when $n \geq 4$.
\end{lemma}

\begin{proof}
    By combining an isotrivial specialization using the $\PGL(3)$ action with semi-continuity, we need only verify this for the pair $(T,T^{*})$ where, in affine coordinates $(x,y)$, $T$ is given by $V(x^2,xy,y^2)$ and $T^{*}$ defines the ``three" concurrent lines consisting of $V(x)$  reckoned 3 times.  Here, the ideal of $\s(T,T^{*})$ is \[(x^{4},x^{3}y,x^{2}y^{2},xy^{3},y^{4},x^3),\] which is clearly generated by polynomials of degree $\leq n$ once $n \geq 4$, proving the result.
\end{proof}

\subsection{A Chern class calculation using localization on $\CT$}
\label{sec:Chern}
Maintain the notation in \Cref{sec:vectorbundle} and define 
\[ E = {p_1}_* \left(\O_{\mathcal{Z}} \otimes p_2^* \O(5)\right).\] Then $E$ is a vector bundle of rank
$9$ on $\CT$ -- indeed, the whole point of introducing $\CT$ in the first
place was to be able to invoke the bundle $E$.

The objective of this section is to evaluate the integer
\[  c_{3}^2(E) - c_{2} (E) c_{4}(E) \in H^6(\CT, \Z) = \Z\] using the
technique of localization.

We recall the localization formula in the context of enumerative
geometry from \cite{ell.str}.  Let $X$ be a smooth projective variety
of dimension $n$ with a $\Gm$ action.  Let $F \subset X$ be the set of
fixed points, and assume that $F$ is finite.  Let $E$ be an
equivariant vector bundle on $X$.  Let $p(c_1, c_2, \dots)$ be a
polynomial in formal variables $c_1, c_2, \dots$.  Assume that $p$ is
weighted homogeneous of degree $n$ when the variable $c_i$ is given
weight $i$.  The goal generally is to compute
\[ p(c_1(E), c_2(E), \dots ) \in H^{2n}(X, \Z) = \Z.\] For $x \in F$,
let $\sigma_i(E, x)$ be the value of the $i$th
elementary symmetric polynomial in the weights of the $\Gm$ acting on
$E|_x$.  Set $f(E, x) = p(\sigma_1(E, x), \sigma_2(E,x), \dots)$.
\begin{theorem}[{Bott's localization formula
    \cite[Theorem~2.2]{ell.str}}]
  \label{thm:localization}
  We have the equality
  \[ p(c_1(E), c_2(E), \dots ) = \sum_{x \in F} \frac{f(E, x)}{\sigma_n(T_X, x)}. \]
\end{theorem}

We now compute the ingredients of the right hand side in
\Cref{thm:localization} for $X = \CT$.  Put homogeneous coordinates
$[X:Y:Z]$ on $\P^2$.  Consider the $\Gm$ action on the $3$-dimensional vector space $\langle X, Y, Z \rangle$ by
\[ t \cdot (X, Y, Z) = (t^a X, t^b Y, t^c Z),\] where
$(a,b,c) \in \Z^3$ are distinct, general integers.  This action of $\Gm$ on
$\P^2$ induces compatible actions on $\check{\P}^{2}$, $\O_{\P^2}(n)$, $\Hilb_3 \P^2$, $\CT$, and
$E$.  Observe that $S_3$ acts on $(X,Y,Z)$ by permutations, and accordingly on $a,b,c$.

\subsection{The 31 fixed points of the $\Gm$ action on $\CT$}
We now list the $S_3$-orbits of the $31$ fixed points of the $\Gm$ action on
$\CT$.  The format is $(I; J)$, where $I \subset k[X,Y,Z]$ is the homogeneous ideal
of a length 3 scheme and $J \subset k[\check{X},\check{Y},\check{Z}]$ is a an ideal describing a length $3$ subscheme of $\check{\P}^{2}$.
\begin{enumerate}
\item $(XY, YZ, XZ; \check{X} \check{Y},\check{X} \check{Z},\check{Y} \check{Z})$ --- the unique 1+1+1 configuration,
\item $(XY, XZ, Y^2; \check{X}\check{Y}, \check{X}\check{Z}, \check{Z}^2)$ and its 6 permutations -- non-linear 1+2 configuration,
\item $(X, Y^2Z; \check{Y}^2 ,\check{Y}\check{Z} , \check{Z}^2)$ and its 6 permutations -- linear 1+2 configuration, 
\item $(X, Y^3; \check{Y}^2 ,\check{Y}\check{Z} , \check{Z}^2)$ and its 6 permutations -- linear 3 configuration,
\item $(X^2, XY, Y^2; \check{Z}, \check{Y}^{3})$ and its 6 permutations -- fat point (non-linear 3 configuration),
\item $(X^2, XY, Y^2; \check{Z}, \check{Y}^{2}\check{X})$ and its 6 permutations -- fat point (non-linear 3 configuration).
\end{enumerate}

This (1)-(6) ordering of $S_3$ orbit representatives will be systematically adhered to for all the following weight computations, including throughout the \texttt{sage} code in \S \ref{sec:sagecode}.

\begin{proposition}
The weights of the $\Gm$ action on $E$ at the fixed points of type (1)-(6) are:
\begin{enumerate}
\item $(5a, 5b, 5c, 4a+b, 4a+c, 4b+a, 4b+c, 4c+a, 4c+b)$,
\item $(5a, 4a+b, 4a+c, 5c, 4c+a, 4c+b, 3c+a+b, 3c+2b, 2c+3b)$ and its 6 permutations,
\item $(5b, 4b+a, 4b+c, 5c, 4c+a, 4c+b, 3c+a+b, 3c+2b, 2c+3b)$ and its 6 permutations,
\item $(5c, 4c+b, 3c+2b, 2c+3b, c+4b, 5b, 4c+a, 3c+a+b, 2c+a+2b)$ and its 6 permutations,
\item $(5c, 4c+a, 3c+2a, 4c+b, 3c+a+b, 2c+2a+b, 3c+2b, 2c+2b+a, 2c+3b)$ and its 6 permutations,
\item $(5c, 4c+a, 3c+2a, 2c+3a, 4c+b, 3c+a+b, 3c+2b, 2c+2b+a, 2c+3b)$ and its 6 permutations.
\end{enumerate}
\end{proposition}

\begin{proof}
In order to indicate how to perform these computations, we will do the case 
(2) as an example -- the reader can then check that no new complications arise in the general calculation.  

Note that in the case (2), the image in
$\Hilb_9 \P^2$ is cut out simply by the square of the ideal; the cubic
 $XY^{2}$ is redundant information.  The length $3$ scheme
$T$ cut out by $\langle XY, XZ, Y^2 \rangle$ is supported at $[1:0:0]$
and $[0:0:1]$.  Near $[1:0:0]$, we can use affine coordinates
$y = Y/X$ and $z = Z/X$.  In these coordinates, the ideal is $(y, z)$,
and its square is $(y^2, yz, z^2)$.  The section $X^5$ of
$\O_{\P^2}(5)$ is $\Gm$ equivariant and non-vanishing at $[1:0:0]$.
We thus get a $\Gm$ equivariant basis of
$H^0(T \cap \{X \neq 0\}, \O_{\s(T,T^{*})}(5))$ given by
\[ X^5 \langle 1, y, z \rangle.\] The corresponding weights are
$5a, 4a+b, 4a+c$.  Near $[0:0:1]$, we can use affine coordinates
$x = X/Z$ and $y = Y/Z$.  In these coordinates, the ideal is
$(x, y^2)$, and its square is $(x^2, xy^2, y^4)$.  The section $Z^5$
of $\O_{\P^2}(5)$ is $\Gm$ equivariant section and non-vanishing at
$[0:0:1]$.  We thus get a $\Gm$ equivariant basis of
$H^0(T \cap \{Z \neq 0\} , \O_{\s(T,T^{*})}(5))$ given by
\[ Z^5 \cdot \langle 1, x, y, xy, y^2, y^3 \rangle.\]
The corresponding weights are $5c, 4c+a, 4c+b, 3c+a+b, 3c+2b, 2c+3b$.
Combining the contributions from $[1:0:0]$ and $[0:0:1]$, we get the full set of weights.
\end{proof}

\begin{proposition}
The weights of the $\Gm$ action on the tangent bundle at the
 fixed points of type (1) - (6) are:
\begin{enumerate}
\item $(c-a, c-b, b-c, b-a, a-b, a-c)$,
\item $(a-c, a-b, c-a, 2c-2b, b-a, c-b)$, and its six permutations,
\item $(c-a, 2c-2b, b-a, c-b, b-c, b-a)$, and its six permutations,
\item $(c-a, 3c-3b, b-a, 2c-2b, 2b-c-a, c-b)$, and its six permutations,
\item $(3b-3a, 2b-2a, b-a, c-a, c-b, a-2b+c)$, and its six permutations,
\item $(a-b, c-b, c-a, 2b-2a, c-b, b-a)$, and its six permutations.
\end{enumerate}
\end{proposition}

\begin{proof}
To see how these are calculated, we consider two representative examples: (3) and (6).  

At a complete triangle of type (3), the
map $\varphi_{1}: \CT \to \Hilb_3 \P^2$ is a local isomorphism.
Therefore, we have
\[ T_p \CT \cong T_p \Hilb_3 \P^2 \cong \Hom_{\P^2} (I, \O_{\P^2}/I).\]
At $[1:0:0]$, we have $I = (y,z)$ where $y = Y/X$ and $z = Z/X$.  So
we get $\Hom(I, \O/I) = \k \langle \widehat y, \widehat z \rangle$
where $\widehat y$ and $\widehat z$ are the dual variables to $y$ and
$z$; their weights are $a-b$ and $a-c$, respectively.  At $[0:0:1]$,
we have $I = (x,y^2)$ where $x = X/Z$ and $y = Y/Z$.  So we get
\[\Hom(I, \O/I) = \Hom_k (\langle x, y^2 \rangle, \langle 1, y \rangle) = \k \langle \widehat x \otimes 1, \widehat x \otimes y, \widehat{y^2} \otimes 1, \widehat{y^2} \otimes y \rangle.\]
The weights of these elements are $c-a, b-a, 2c-2b, c-b$.  Combining
the contributions from $[1:0:0]$ and $[0:0:1]$, we get the full set of
weights.

Now we shift focus to a fixed point of type (6).  At this point,
the map $\CT \to \Hilb_3 \check{\P}^2$ is a local isomorphism.  Denote
by $\widehat{X}$, $\widehat{Y}$, and $\widehat{Z}$ the variables dual
to $X$, $Y$, and $Z$.  Then the corresponding point in
$\Hilb_3 \check{\P}^2$ is the point
\[ \langle \widehat{Z}, \widehat{Y}^2\widehat{X} \rangle.\] As in the
third case, we get the weights $(c-a, 2b-2a, c-b, b-a, a-b, c-b)$.

\end{proof}

We can now compute top degree Chern expressions of the vector bundle $E$. In particular, using \Cref{thm:localization}, we get

\begin{equation}
  \label{eq:57728}
c_3(E)^2 - c_2(E)c_4(E) = 2^7 \times 11 \times 41 = 57728.
\end{equation}
The computation is carried out in \S \ref{sec:sagecode}.

Before continuing with our story, we take a moment to collect other weight calculations which will be needed in the last section.  We maintain the 1-6 ordering of the $S_{3}$ orbits of fixed points throughout. 
\begin{proposition}
\label{proposition:weightstautological}
Let $\O^{[3]}, \O(1)^{[3]},$ and $\O(2)^{[3]}$ denote the tautological rank $3$ bundles pulled back to $\CT$, and let $\O(H)$ denote the line bundle on $\CT$ corresponding to the divisor of those $(T,T^{*})$ such that $T$ meets a fixed line. The torus weights of these bundles at the fixed points in $\CT$ are listed below.  (One then applies the action of $S_3$ on the letters $\{a,b,c\}$.) 
\begin{itemize}
    \item[$\O^{[3]}$:] \begin{enumerate}
        \item $(0,0,0)$
        \item $(0,b-c,0)$,
        \item $(0,b-c,0)$,
        \item $(0,b-c,2b-2c)$,
        \item $(0,a-c,b-c)$,
        \item $(0,a-c,b-c)$.
    \end{enumerate}
    \item[$\O(1)^{[3]}$:] \begin{enumerate}
        \item $(a,b,c)$,
        \item $(a,b,c)$,
        \item $(b,b,c)$,
        \item $(c,b,2b-c)$,
        \item $(a,b,c)$,
        \item $(a,b,c)$.
    \end{enumerate}
    \item[$\O(2)^{[3]}$:] \begin{enumerate}
        \item $(2a,2b,2c)$,
        \item $(2a,b+c,2c)$,
        \item $(2b,b+c,2c)$,
        \item $(2b,b+c,2c)$,
        \item $(a+c,b+c,2c)$,
        \item $(a+c,b+c,2c)$.
    \end{enumerate}
    \item[$\O(3)^{[3]}$:] \begin{enumerate}
        \item $(3a,3b,3c)$,
        \item $(3a,b+2c,3c)$,
        \item $(3b,b+2c,3c)$,
        \item $(b+2c,2b+c,3c)$,
        \item $(a+2c,b+2c,3c)$,
        \item $(a+2c,b+2c,3c)$.
    \end{enumerate}
    \item[$\O(H)$:] \begin{enumerate}
        \item $(a+b+c)$,
        \item $(a+2c)$,
        \item $(b+2c)$,
        \item $(3c)$,
        \item $(3c)$,
        \item $(3c)$. 
    \end{enumerate}
\end{itemize}
\end{proposition} 

\section{Circumventing excess}
\label{sec:pencilspace}

\subsection{Why 57728 is wrong}

The calculation \eqref{eq:57728} done in the previous section is unfortunately not the number $\nu_{3,2}$.  The problem is that the  evaluation map 
\begin{align}
    \label{eq:ev}
\underline{H^{0}(\P^{2}, \I_{\Gamma_{13}}(5))} \to E,
\end{align}
where $\Gamma_{13} \subset \P^{2}$ is a general set of $13$ points (see \Cref{subsec:returning}), has $2$-dimensional kernel over certain points $(T,T^{*}) \in \CT$ which should not be counted for our enumerative problem -- there is excess in the degeneracy scheme.  For a simple example, observe that if $T$ consists of three of the points of $\Gamma_{13}$, then the evaluation mapping over $(T,T^{*})$ automatically has at least a $2$-dimensional kernel. There are more complicated contributions to the excess. (Ultimately this is explained by \Cref{lemma:threecomponents} below.) To dodge the excess, we change our viewpoint and work in a Grassmannian bundle over $\CT$ which we denote by $\SQP$ and call \emph{the space of singular quintic pencils}.  Recall that the vector bundle $\mathcal{V}_{5}$ (see \Cref{proposition:vectorbundle}) is related to $E$ by an exact sequence $$0 \to \mathcal{V}_{5} \to \underline{H^{0}(\P^{2}, \O(5))} \to E \to 0$$ over $\CT$.

\begin{definition}
  \label{definition:quinticsingularpencils} The {\bf space of singular quintic pencils}, denoted  $\SQP$, is the
  smooth variety $\Gr(2,\mathcal{V}_{5})$ representing $2$-dimensional subspaces in the fibers of $\mathcal{V}_{5}$.
\end{definition} 
A point of $\SQP$ corresponds to a triple $(T, T^{*}, \Lambda)$
where $(T,T^{*}) \in \CT$ is a complete triangle and $\Lambda$ is a
pencil of quintic curves, all containing the length 9 scheme
$\s(T,T^{*})$. We let \[\varphi: \SQP \to \CT\] denote the map sending
$(T,T^{*},\Lambda)$ to $(T,T^{*})$;  $\varphi$ is a $\Gr(2,12)$-bundle.  The variety $\SQP$ also affords a natural map
\[\eta: \SQP \to \Gr\left(2, H^{0}(\P^{2}, \O(5))\right)\]
with formula $\eta(T,T^{*},\Lambda) = \Lambda$. 

Let us state the main objective of this section from the outset, to better orient the reader: 

\begin{theorem}
    \label{theorem:Domintegral} Let $p \in \P^{2}$ be a point, and let $\Dom(p) \subset \SQP$ be as in \Cref{definition:D} below. Then $$\nu_{3,2} = \int_{\SQP} \left[\Dom(p)\right]^{13}.$$
\end{theorem}
We begin by establishing some definitions.

\subsection{First definitions}

\begin{definition}
    \label{definition:base} Let $W \subset H^{0}(\P^{2},\O(n))$ be a subspace. The {\bf base-scheme} of $W$, denoted $\Base(W)$ is the subscheme of $\P^{2}$ which is the common vanishing scheme of all elements of $W$.
\end{definition}

\begin{definition}
  \label{definition:finite} We let \[\Inf \subset \SQP\] denote the
  closed subset consisting of triples $(T,T^{*},\Lambda)$ such that $\Base(\Lambda)$ is infinite. We let
  \[\Fin \subset \SQP\] denote the open complement of $\Inf$.
\end{definition}

\subsection{Point conditions on singular quintic pencils}
\label{sec:pointconditionSQP}

Fix a point $p \in \P^{2}$.  The point $p$ determines the hyperplane
\[H_{p} \subset H^{0}(\P^{2}, \O(5))\] consisting of quintic forms vanishing at $p$,
and therefore also determines a codimension $2$ sub-Grassmannian
\[\Gr\left(2, H_{p}\right) \subset \Gr\left(2, H^{0}\left(\P^{2}, \O(5)\right)\right).\]

\begin{definition}
  \label{definition:basepointp} Maintain the setting immediately
  prior. We define \[\BP (p) \subset \SQP\] to be
  the subscheme $\eta^{-1}\left(\Gr(2,H_p)\right)$.
\end{definition}
As a set, $\BP (p)$ consists of those triples
$(T,T^{*},\Lambda) \in \SQP$ having the property that $p \in \Base(\Lambda)$.

\begin{proposition}
  \label{proposition:BPp}
  $\BP (p)$ is a codimension 2 local complete intersection subscheme
  of $\SQP$.
\end{proposition}

\begin{proof}
    Since $\Gr(2,H_{p})$ is smooth and is a codimension $2$ subvariety of $\Gr\left(2, H^{0}(\P^{2}, \O(5))\right)$, and since $\SQP$ is also smooth, it follows that $\codim \BP(p) \leq 2$, and it suffices to show $\codim \BP(p) = 2$. It is clear that $\BP (p) \neq \SQP$, so let us assume for sake of contradiction that $C \subset \BP(p)$ is an irreducible component which is a divisor in $\SQP$.  

    Dimension constraints yield two possibilities. The first possibility is $\varphi(C) = \CT$, and the second possibility is that $\varphi(C) \subset \CT$ is a divisor.  Before continuing the proof, we prove a general claim:  
    \begin{lemma}
    \label{claim:support}
    Let $(T,T^{*}) \in \CT$ be arbitrary, and let $W = H^{0}(\P^{2}, \I_{\s(T,T^{*})}(5))$.  Then \[\Supp \Base(W) = \Supp T\] if and only if $T \notin \Thin$ and otherwise $\Supp \Base(W)$ is the line $\langle T \rangle$.
    \end{lemma}
    \begin{proof}[Proof of \ref{claim:support}]
    The claim rests on the observation that $\Supp T = \Supp \Base (\Lambda_{T})$ if and only if $T$ is not contained in a line. Here $\Lambda_{T}$ is the net of conics (\autoref{definition:NetZ}) containing $T$.  If $T \in \Thin$, then the quintic curves in $W$ are those which contain the line spanned by $T$ as a component, and whose residual quartic curve contains $T$.  From this description it is clear that $\Supp\Base(W) = \langle T \rangle$ as claimed.  
    
    So we can and will assume $T \notin \Thin$. Then the inclusion $$\Supp \Base (H^{0}(\I_{T}^{2}(4))) \subset \Supp T$$ is seen by considering the pairwise products of  three quadratic polynomials spanning $\Lambda_{T}$.  Therefore, $\Supp \Base (H^{0}(\I_{T}^{2}(5))) \subset \Supp T$ as well. Since $\Base (W) \subset \Base (H^{0}(\I_{T}^{2}(5)))$, we conclude that $\Supp \Base (W) \subset \Supp T$. The opposite inclusion is trivial, concluding the proof. 
    \end{proof}

    Returning to the proof of \Cref{proposition:BPp}, we consider the case $\varphi(C) = \CT$.  Choose a general honest triangle $(T,T^{*})$ -- in particular, $p \notin T$.  Then, from \Cref{claim:support} we know that $T = \Supp \Base(W)$.  Thus, the point $p$ imposes a non-trivial condition on the elements of $W$, i.e. the vector space $V \subset W$ consisting of those $w \in W$ satisfying $w(p)=0$ has codimension $1$.  Therefore $\varphi^{-1}(\{(T, T^{*})\}) \cap C \subset \Gr(2,V)$ has codimension at least $2$ in $\varphi^{-1}(\{(T, T^{*})\})$, and hence $C$ cannot be a divisor in $\SQP$, our desired contradiction.  

    It remains to deal with the possibility that $\varphi(C)$ is a (irreducible) divisor in $\CT$. Let $(T,T^{*})$ be a general element of $\varphi(C)$, assuming it is a divisor. In particular, $p \notin T$ and if $(T,T^{*}) \in \Thin$ then $p \notin \langle T \rangle$, as otherwise $\varphi(C)$ would have codimension strictly larger than $1$.  By \Cref{claim:support}, it follows that the space of sections of $W$ vanishing at $p$ is a proper subspace $V \subset W$.  Since $\varphi^{-1}(\{(T,T^{*})\}) \cap C \subset \Gr(2,V)$, and $\Gr(2,V)$ has codimension $2$ in $\Gr(2,W)$, it follows that $C$ has codimension at least $3$ in $\SQP$, a contradiction.  The proposition is now proved.
\end{proof}

The proof of \Cref{proposition:BPp} justifies the following definition.

\begin{definition}
  \label{definition:D}
  \begin{enumerate}
    \item    We define \[\Dom (p) \subset \BP (p)\] to be the unique irreducible
  component (with reduced, induced scheme structure) whose general
  point $(T,T^{*},\Lambda)$ satisfies
    \begin{enumerate}
     \item $(T,T^{*})$ is an honest triangle, and
     \item $p \notin T$.
      \end{enumerate} 

    \item  We define \[ \Inc (p) \subset \BP (p) \] to be the irreducible component 
  (with reduced, induced scheme structure) consisting of triples
  $(T,T^{*},\Lambda)$ such that $p \in T$.
    \item  We define \[\Lin (p) \subset \BP (p)\] to
  be the irreducible component (with reduced, induced scheme
  structure) whose general point corresponds to a triple $(T,T^{*}, \Lambda)$ satisfying
  \begin{enumerate}
  \item $T \in \Thin$
  \item $p \notin T$, and
  \item $T$ and $p$ are collinear.
  \end{enumerate}
  \end{enumerate}
\end{definition}

\begin{remark}
    \label{remark:Dombundle} Observe that if $(T,T^{*})$ is such that $p \notin T$ and if $\s(T,T^{*}) \cup \{p\}$ imposes ten independent conditions on quintic forms, then $\Dom(p)$ is the unique irreducible component of $\BP(p)$ lying over a sufficiently small neighborhood of $(T,T^{*})$.  Indeed, over a neighborhood of $(T,T^{*}) \in \CT$ the map $\eta|_{\BP(p)}: \BP(p) \to \CT$ is then a $\Gr(2,11)$-bundle.
\end{remark}

\begin{lemma}
    \label{lemma:threecomponents} $\Dom(p), \Inc(p),$ and $\Lin(p)$ are the irreducible components of $\BP(p)$.
\end{lemma}

\begin{proof}
    Let $Z \subset \BP(p)$ be an irreducible component, and let $(T,T^{*},\Lambda)$ be a general element of $Z$.  We must show $Z$ is one of the three listed sets.
    
    If $p \in T$, then $Z = \Inc(p)$ (for dimension reasons) and we are done.  So we may and will assume $p \notin T$ from here.  

    If $T \in \Thin$, then $\Lambda$ must be a pencil of quintics which has the line $\langle T\rangle$ in its base scheme. If $p \in \langle T \rangle$ then again by counting dimensions we conclude $Z = \Lin(p)$.  If $p \notin \langle T \rangle$ on the other hand, then the set of all such $(T,T^{*},\Lambda)$ does not have large enough dimension to contribute an irreducible component of $\BP(p)$.

    If $(T,T^{*})$ is such that $T \notin \Thin$ then by \Cref{lemma:generatedquintics} the twisted ideal sheaf $\I_{\s(T,T^{*})}(5)$ is globally generated, and therefore $\s(T,T^{*}) \cup \{p\}$ imposes ten independent conditions on quintic forms.  Therefore, $Z = \Dom(p)$ by \Cref{remark:Dombundle}, finishing the proof.    
\end{proof} 

Our next objective is to determine the multiplicities of $\BP(p)$ along its three irreducible components $\Dom(p), \Inc(p),$ and $\Lin(p)$ (\Cref{lemma:threecomponents}).   We need a bit of preparation before continuing.  If we let $S \to \SQP$ denote the universal rank $2$ vector bundle, then we can let  

\begin{eqnarray}
    S^{\dagger} \subset S \times _{\SQP}S
\end{eqnarray}
denote the bundle of frames for $S/\SQP$.  A point of $S^{\dagger}$ is a tuple $(F,G,T,T^{*})$ where $F$ and $G$ are linearly independent quintic forms which are both elements of $H^{0}\left(\P^{2}, \I_{\s(T,T^{*})}(5)\right)$.  The natural morphism $S^{\dagger} \to \SQP$ is smooth and faithfully flat (it is a $\GL_{2}$-torsor), and therefore, if we use the $\dagger$ superscript in the obvious way, it suffices to study the multiplicities of $\Dom(p)^{\dagger}, \Inc(p)^{\dagger},$ and $\Lin(p)^{\dagger}$ as irreducible components of $\BP(p)^{\dagger}$.  

The reason for passing to $S^{\dagger}$ is that $\BP(p)^{\dagger}$ is a global complete intersection of two divisors: If we define 
\begin{eqnarray*}
    H_{1}(p) &:= \left\{\begin{tabular}{c}
         $(F,G,T,T^{*}) \in S^{\dagger}$ such\\ that
         $F(p)=0$.\\ 
    \end{tabular} \right\} \text{ and}\\
    H_{2}(p) &:= \left\{\begin{tabular}{c}
         $(F,G,T,T^{*}) \in S^{\dagger}$ such\\ that
         $G(p)=0$.\\ 
    \end{tabular} \right\} \text{,}
\end{eqnarray*}
    then $\BP(p)^{\dagger} = H_{1}(p) \cap H_{2}(p)$ as schemes. 
 For transversality arguments, it behooves us to better illuminate the three tangent spaces of $S^{\dagger}, H_{1}(p),$ and $H_{2}(p)$ at a point $(F,G,T,T^{*})$.  We will only need to analyze the situation where $T \subset \P^{2}$ consists of three distinct points $\{a,b,c\}$, and where $F$ and $G$ possess only ordinary nodes at $a,b,c$.  We will use the language of deformation theory, working over the dual numbers $\k[\varepsilon]/(\varepsilon^{2})$. 
    
    A first order deformation of $F$ (resp. $G$) is given by $F + \varepsilon F'$ (resp. $G + \varepsilon G'$) where $F'$ (resp. $G'$) is a quintic form.  A first order deformation of $F$ (resp. $G$) which continues to have three nodes and which  \emph{allows their locations  $a,b$ and $c$ to deform} is of the form $F + \varepsilon F'$ where $F'(a)=F'(b) = F'(c) = 0$. Therefore, the tangent space $T_{(F,G,T,T^{*})}S^{\dagger}$ is a linear subspace of the product vector space $H^{0}\left(\P^{2},\I_{T}(5)\right) \times H^{0}\left(\P^{2},\I_{T}(5)\right)$, one  which we identify next.

    The essential question we must answer is: Given a form $F'$ vanishing at $a,b,$ and $c$, how do we determine the deformation of the nodes $a,b,c$ induced by $F + \varepsilon F'$? A local calculation reveals a clean answer. We will focus on just the point $a$. $F$ is a global section of the sheaf $\m_{a}^{2}(5)$, and therefore induces an element $$\hess_{F} \in \left(\m_{a}^{2}/\m_{a}^{3}\right) \otimes \O(5).$$ By the natural isomorphism $$\m_{a}^{2}/\m_{a}^{3} \simeq \Sym^{2}(\m_{a}/\m_{a}^{2}),$$ and in light of the inclusion (char. $\k \neq 2$) $$\Sym^{2}(\m_{a}/\m_{a}^{2}) \subset \Hom_{\k}\left(\left(\m_{a}/\m_{a}^{2}\right)^{\vee} , \m_{a}/\m_{a}^{2} \right),$$ we may view $\hess_{F}$ as an element of $$\Hom \left(\left(\m_{a}/\m_{a}^{2}\right)^{\vee} , \left(\m_{a}/\m_{a}^{2}\right) \otimes \O(5) \right).$$  The notation is chosen because in local coordinates if $f(x,y)$ is an affine quintic obtained by dehomogenizing $F$ then $\hess_{F}$ is represented by the $2 \times 2$  \emph{Hessian matrix} $$\begin{bmatrix}
        f_{xx} & f_{xy} \\
        f_{xy} & f_{yy} \\
    \end{bmatrix}.$$   As a consequence of $F$ having an ordinary node at $a$, we see that $\hess_{F}$ is invertible.  
    
    Now, suppose $F + \varepsilon F'$ is a deformation satisfying $F'(a) = 0$. Then $F'$ determines its differential $$d F'_{a} \in \left(\m_{a}/\m_{a}^{2}\right) \otimes \O(5),$$ and therefore $$\tau_{F,a}(F') := \hess_{F}^{-1}(d F'_{a})$$ is a well-defined element of the tangent space $$\left(\m_{a}/ \m_{a}^{2} \right)^{\vee} = T_{a}\P^{2}.$$ This vector is the deformation of the node at  $a$ induced by the deformation $F + \varepsilon F'$, as can easily be checked in local coordinates.  We are prepared to  express the tangent space $T_{(F,G,T,T^{*})}S^{\dagger}$: 
    \begin{eqnarray}
    \label{tangentSdagger} T_{(F,G,T,T^{*})}S^{\dagger} = \left\{ \begin{tabular}{c}
            Pairs of quintics $(F',G')$ vanishing \\
            at $T$ 
            satisfying $\tau_{F,t}(F') = \tau_{G,t}(G')$ for \\ all three points $t \in T$
        \end{tabular}
        \right\}.
    \end{eqnarray}

\begin{lemma}
    \label{LinTransverse} The divisors $H_{1}(p)$ and $H_{2}(p)$ intersect transversely at a general point of $\Lin(p)^{\dagger}.$
\end{lemma}

\begin{proof}
    We will  exhibit a single point $$(F,G,T,T^{*}) \in \Lin(p)^{\dagger}$$ where the two tangent spaces $T_{(F,G,T,T^{*})}H_{1}(p)$ and $T_{(F,G,T,T^{*})}H_{2}(p)$ are distinct, codimension 1 spaces of the ambient tangent space $T_{(F,G,T,T^{*})}S^{\dagger}$.

    Consider then, the following tuple $(F,G,T,T^{*})$ and point $p$: 
    \begin{align*}
        F & = XY(Y-Z)(Y-\lambda Z)Z, \\
        G & = XY(Y-Z)(Y-\mu Z)Z ,\\
        T & = \left\{a = [0:0:1], b = [0:1:1], c = [0:1:0] \right\},\\
        p & = [0:\lambda:1],
    \end{align*}
    where $\lambda, \mu \in \k$ are to be chosen generally. $(F,G,T,T^{*})$ is evidently contained in $\Lin(p)^\dagger$: all points $a,b,c,p$ lie on the line $X=0$. Additionally, $F$ and $G$ have only simple nodes at the points $a,b,c$, and so the description in \eqref{tangentSdagger} of $T_{(F,G,T,T^{*})}S^{\dagger}$ as a vector space of certain pairs $(F',G')$ applies.   And so, consider the pair 
    \begin{align*}
        F' = Y(Y-Z)(Y-\lambda Z)Z^2,\\
        G' = Y(Y-Z)(Y-\mu Z)Z^{2}.
    \end{align*}
    A local calculation (omitted) then shows that the three membership conditions of \eqref{tangentSdagger}, namely $$\tau_{F,t}(F') = \tau_{G,t}(G'), \,\, \forall t \in T,$$ are met. Furthermore, $(F',G')$ is evidently contained in $T_{(F,G,T,T^{*})}H_{1}(p)$ (because $F'(p)=0$) and is not contained in $T_{(F,G,T,T^{*})}H_{2}(p)$ (because $G'(p) \neq 0$). The lemma follows.  
\end{proof}

\begin{theorem}
  \label{theorem:BPdecomposition}
  As codimension $2$ cycles in $\SQP$, 
  \begin{eqnarray}
      \left[\BP(p) \right] = \left[\Dom(p) \right] + 4\left[\Inc(p)
    \right] + \left[\Lin(p) \right].
  \end{eqnarray}
\end{theorem}

\begin{proof}
    We must explain the multiplicities. Since the natural map $S^{\dagger} \to \SQP$ is smooth and surjective, it suffices to prove 
    \begin{align}
    \label{daggermult}
    \left[\BP(p)^{\dagger} \right] = \left[\Dom(p)^{\dagger} \right] + 4\left[\Inc(p)^{\dagger}
    \right] + \left[\Lin(p)^{\dagger} \right]
    \end{align} as codimension 2 cycles in the frame bundle  $S^{\dagger}$.   The coefficient of $\Lin(p)^{\dagger}$ is explained by \Cref{LinTransverse}. We will only  focus on the coefficient $4$ of $\Inc(p)^{\dagger}$, as the ideas apply equally well (and with fewer complications) to the  coefficient of $\Dom(p)^{\dagger}$. 

    We will choose a sufficiently general $2$-dimensional \'etale-local slice of $S^\dagger$ at a general point $(F,G,T,T^{*}) \in \Inc(p)^{\dagger}$. So, let $(U,q) \subset S^{\dagger}$ be a smooth pointed surface with \'etale-local coordinates $s,t$ at $q$, and suppose $q = (F,G,T,T^{*})$ is a general point of $\Inc(p)^{\dagger}$. Observe that $T = \{p, t_{2}, t_{3}\}$ is an honest triangle, so there is an \'etale neighborhood $V$ of $(T,T^{*}) \in \CT$  which is isomorphic to an \'etale neighborhood the point $(p,t_{2},t_{3}) \in (\P^{2})^{3}$. Let $\alpha_{1}:V \to \P^{2}$ denote projection onto the first factor.  Finally, as part of the generic hypotheses on $U$, after possibly shrinking $U$, suppose the composite $U \to V \to \P^{2}$ is unramified at $q$.  
    
    Having made the choice of the general slice $U$, our objective is to understand the two curves $H_{i}(p) \cap U, i=1,2$ locally near $q$.  Dehomogenizing the family of forms parametrized by $U$, and letting $p = (0,0) \in \mathbb{A}^{2}$, we obtain a pair of varying polynomials dependent on $(s,t)$: 
    \begin{align*}
        f_{(s,t)}(x,y) = c_{11}(x-u)^{2} + c_{12}(x-u)(y-v) + c_{22}(y-v)^{2} + \dots \\
        g_{(s,t)}(x,y) = d_{11}(x-u)^{2} + d_{12}(x-u)(y-v) + d_{22}(y-v)^{2} + \dots,
    \end{align*}
    where $u = u(s,t)$ and $v = v(s,t)$ are the coordinates of the particular node which coincides with the point $p$ at $s=t=0$. The coefficients $c_{ij},d_{ij}$ are also functions of $(s,t)$, and when $s=t=0$ we may assume that $f_{(0,0)}$ and $g_{(0,0)}$ are tri-nodal quintics whose tangent cones at $p=(0,0)$ do not share a line. 

    Now, by introducing the condition ``$H_{1}(p)$" we are simply plugging in $x=y=0$ into $f_{(s,t)}$ and requesting vanishing. This gives the equation $$c_{11}u^2 + c_{12}uv + c_{22}v^2 + \dots = 0,$$ where the excluded terms lie in $(s,t)^{3}$. We obtain a similar local equation for $H_{2}(p) \cap U$ with $c_{ij}$'s replaced with $d_{ij}$'s.  The two functions $u,v$ can be taken to be local coordinates at $q$ -- this is due to the genericity assumptions on the slice $U$. It follows that the local equations $H_{i}(p) \cap U$ inherit the property of having ordinary nodes at $q$ from the fact that $f_{(0,0)}$ and $g_{(0,0)}$ both had ordinary nodes. Thus, $H_{1}(p) \cap U$ and $H_{2}(p)\cap U$ are two curves nodal at $q$.  The tangent cones of $H_{1}\cap U$ and $H_{2} \cap U$ at $q$ do not share a line because the same was true for the curves $f_{(0,0)}$ and $g_{(0,0)}$.  Thus,  the multiplicity $4$ occurring in \eqref{daggermult} is explained.  
\end{proof}

%%%
%%%\begin{lemma}
  %\label{IncDom} Let $(T,T^{*},\Lambda) \in \Dom(p) \cap \Inc(p)$, and
  %let $Z \subset s(T,T^{*})$ denote the closed subscheme supported at
  %$p$. Then there exists a 0-dimensional scheme $Z'$ properly containing $Z$
  %and supported at $p$, such that $Z' \subset \Base(\Lambda)$.
%%%\end{lemma}

\subsection{Thirteen point conditions}
\label{sec:thirteenpointconditions}

From here on, let $\Gamma_{13} = \{p_{1}, \dots, p_{13}\}$ be 13 general points in
$\P^{2}$, and set \begin{equation}
\label{equation:Theta}\Theta := \bigcap \limits_{i=1, \dots, 13} \Dom(p_i).
\end{equation}
Our next major objective is \Cref{theorem:solutionset}, which states that $\Theta$ is indeed the set we must enumerate. The strategy is of course to argue that certain possible and undesirable types of points do not occur in the intersection.  We first deal with possibilities inside the open set $\Fin$ using dimension counts in  \Cref{lemma:finite}, and \Cref{lemma:finitehonest}. Then we deal with undesirable possibilities inside $\Inf$ -- this is more difficult, occupying \Cref{lemma:QLNo}, \Cref{proposition:infline},  \Cref{proposition:Movesingular}, and \Cref{proposition:LinDom}.  \Cref{proposition:LinDom} in particular uses a limit-linear series argument.

\subsubsection{Finite base schemes}
\label{sec:understanding-s-cap}

\begin{lemma}
  \label{lemma:finite} Every point
  $(T, T^{*}, \Lambda) \in \Theta \cap \Fin$ satisfies:
  $(T,T^{*},\Lambda) \notin \Inc(p_{i})$ and
  $(T,T^{*},\Lambda) \notin \Lin(p_{i})$ for all
  $p_{i} \in \Gamma_{13}$.
\end{lemma}

\begin{proof}
  This follows from combining a dimension count with the condition
  that $\Gamma_{13}$ is a general set. We only prove the $\Inc$
  statement -- the other case proceeds mutatis mutandis.

  Without losing generality, we need only prove the statement
  $(T,T^{*},\Lambda) \notin \Inc(p_{1})$.  As $\Dom(p_1)$ is
  irreducible and $24$-dimensional,
  \[\dim \Dom(p_1) \cap \Inc(p_1)  \leq 23.\] And so, the locally closed
  set $U := \Dom(p_1) \cap \Inc(p_1) \cap \Fin$ is at most $23$-dimensional.

  Every $(T,T^{*},\Lambda) \in U$ determines the $0$-dimensional, length
  $25$ scheme $\Base(\Lambda) \subset \P^{2}$. Let
  \[U' \subset U \times (\P^{2})^{12}\] denote the set
  parametrizing quadruples $(T,T^{*},\Lambda,W)$ where
  $(T,T^{*},\Lambda) \in U$ and where $W = (w_1, \dots, w_{12})$ is a
  $12$-tuple of points in $\P^{2}$ satisfying $w_i \in \Base(\Lambda)$
  for all $i$. Then the forgetful map $U' \to U$ is quasi-finite, and
  hence \[\dim U' \leq 23.\]

  The map $U' \to (\P^{2})^{12}$ defined by
  $(T,T^{*}, \Lambda, W) \mapsto W$ therefore cannot dominate
  $(\P^{2})^{12}$, and hence the set of $13$ \emph{general} points
  $\{p_1, \dots, p_{13}\}$ cannot be contained in $\Base(\Lambda)$ for
  $(T,T^{*},\Lambda) \in \bigcap_{i=1}^{13}\Dom(p_{i}) \cap
  \Inc(p_{1}) \cap \Fin$, implying the lemma.
\end{proof}

\begin{lemma}
  \label{lemma:finitehonest}
  Let $(T,T^{*},\Lambda) \in \Theta \cap \Fin$. Then $T$ is an honest
  triangle, and $T \cap \Gamma_{13} = \varnothing$.
\end{lemma}

\begin{proof}
  The statement ``$T \cap \Gamma_{13} = \varnothing$" follows
  immediately from \Cref{lemma:finite}, so we will assume it and
  show that $T$ is an honest triangle.

  Pick $(T,T^{*},\Lambda) \in \Theta \cap \Fin$. As $\Base(\Lambda)$ is
  finite, it follows that $T \notin \Thin$ (as otherwise $\Base(\Lambda)$
  contains the line spanned by $T$). Therefore, we need only show that $T$ is reduced.
  
  As in the proof of \Cref{lemma:finite}, we perform a
  dimension count.  The locus $V \subset \Fin$ consisting of triples
  $(T,T^{*},\Lambda)$ where $T$ is non-reduced is a Cartier
  divisor. And so, $\dim V = 25$. Let
  \[V' \subset V \times (\P^{2})^{13}\] denote the scheme
  parametrizing quadruples $(T,T^{*},\Lambda, W)$, where
  $(T,T^{*},\Lambda) \in V$ and $W = (w_1, \dots, w_{13})$ is a 
  $13$-tuple of points satisfying $w_{i} \in \Base(\Lambda)$ for all
  $i$. Since the forgetful map $V' \to V$ is quasi-finite, it follows
  that $\dim V' = 25$. Therefore the second projection
  $V' \to (\P^{2})^{13}$ cannot be dominant, and in particular cannot
  contain the general tuple $(p_1,\dots, p_{13})$ in its image, implying the lemma.
\end{proof}

\subsubsection{Dealing with infinite base schemes}
\label{sec:inf-considerations}

We now take on the challenge of showing emptiness of $\Theta \cap \Inf$.

\begin{definition}
  \label{definition:1dpart} 
  \begin{enumerate}
    \item If $\Lambda \subset H^{0}(\P^2, \O(d))$ is
    a pencil of degree $d$ curves, with base scheme $\Base(\Lambda)$, we
    define the {\bf fixed curve} of $\Base(\Lambda)$ to be the Cartier
    divisor on $\P^{2}$ defined by the greatest common factor of any two
    general elements of $\Lambda$. 
    \item If the fixed curve of $\Lambda$ has
    degree $e$, we define the {\bf moving part} of $\Lambda$, denoted $\Lambda'$, to be the
    pencil of degree $d-e$ curves obtained by dividing the equations of
    the members of $\Lambda$ by the equation of the fixed curve.
    \item If $\Lambda$ is a pencil of curves with fixed curve $C$, then a {\bf fixed point} of $\Lambda$ will mean a point in $C \setminus \Base(\Lambda')$.
    \item If $\Lambda$ is a pencil of curves with fixed curve $C$, then an {\bf isolated point} of $\Lambda$ will mean a point in $\Base(\Lambda') \setminus C$.
    \item If $\Lambda$ is a pencil of curves with fixed curve $C$, then an {\bf embedded point} of $\Lambda$ will mean a point in $C \cap \Base(\Lambda')$.
  \end{enumerate}
\end{definition}

\begin{lemma}
    \label{lemma:QLNo}
    Suppose $\Lambda$ is a pencil of quintic curves whose fixed curve is a quartic $Q$.
    Suppose furthermore that $\Lambda$ satisfies one of the following: 
    \begin{enumerate}
        \item[(a)] $\Base(\Lambda)$ has no embedded points and $\Sing Q$ consists of at most two points, each a node,
        \item[(b)] $\Base(\Lambda)$ has no embedded points and $\Sing Q$ consists of a single ordinary cusp,
        \item[(c)] $\Base(\Lambda)$ has one embedded point at a smooth point of $Q$ while $\Sing Q$ has at most one singular point which is a node,
        \item[(d)] $\Sing Q$ consists of a single node and this node is the embedded point of $\Base(\Lambda)$.
    \end{enumerate}

    Then $\Base (\Lambda)$ does not contain any subscheme of the form $\s(T,T^{*})$ for any $(T,T^{*}) \in \CT$.
\end{lemma}

\begin{proof}
    The moving part $\Lambda'$ of $\Lambda$ is a pencil of lines; we let $b \in \P^{2}$ denote the base point $\Lambda'$.  Then, $\Base (\Lambda)$ contains an embedded point if and only if $b \in Q$, in which case $b$ is the sole embedded point. Before taking each case up in turn, observe that as $\s(T,T^{*})$ has $2$-dimensional Zariski tangent space at each of its points, if $\s(T,T^{*}) \subset \Base(\Lambda)$ then $T$ must be supported on the singular points of $Q$ or on the embedded point of $\Base(\Lambda)$, if it exists (or both).
    \begin{enumerate}
        \item[(a)] In this scenario, one of the two nodes, call it $n \in Q$, must support a length $\geq 2$ connected component $T_{n} \subset T$ of $T$. 
        
        If $T_{n}$ has length $2$, then in affine coordinates $(x,y)$ around $n$ the ideal $\I_{T_{n}}$ is $(y,x^{2})$ and so the ideal $\I_{\s(T,T^{*})}$ is given by $(y^{2},yx^{2},x^{4})$.  The ideal $\I_{\s(T,T^{*})}$ does not contain an element with non-degenerate quadratic part, and so $Q$'s local equation cannot be contained it, eliminating this case.  
        
        If $T_{n}$ has length $3$ and is curvilinear, then in suitable \emph{analytic} local coordinates $(u,v)$ around $n$, the ideal $\I_{T_{n}}$ can be taken to be $(u,v^{3})$, and so $\I_{\s(T,T^{*})} = (u^{2},uv^{3},v^{6})$.  Once again, an analytic local equation for $Q$ cannot  be contained in $\I_{\s(T,T^{*})}$ because it would have a non-degenerate quadratic term. 
        
        Finally, if $T_{n}$ is a fat point, then $\s(T,T^{*})$ (for any $T^{*}$) is contained in the cube of the maximal ideal $\m \subset \O_{\P^{2},n}$, while $Q$'s local equation lies in $\m^{2} \setminus \m^{3}$.  And so in all possible scenarios, $Q$'s local equation cannot be contained in $\I_{\s(T,T^{*})}$, which is what we needed to show.
        
        \item[(b)] Let $c \in Q$ denote the cusp. If $\s(T,T^{*}) \subset \Base(\Lambda)$ then $T$ is entirely supported on $c$, and hence $T$ is either fat or curvilinear. 
        
        If $T$ is a fat point, then $\I_{\s(T,T^{*})}$ is contained in the cube of the maximal ideal $\m \subset \O_{\P^{2},c}$, yet $Q$'s local equation is contained in $\m^{2} \setminus \m^{3}$ (a cusp is a double point) -- so the fat possibility is eliminated. 
        
        If $T$ is curvilinear, then in suitable analytic local coordinates $(u,v)$ around $c$, we have $\I_{T} = (u,v^{3})$ and so $\I_{\s(T,T^{*})} = (u^{2},uv^{3},v^{6})$. Suppose $g(u,v)$ is an analytic local equation for $Q$. If $g \in \I_{\s(T,T^{*})}$ then, because a cusp is a double point, after scaling by an element in $\k^{\times}$ we must have $$g = u^{2} + h_1(u,v) \cdot u^{2} + h_2(u,v) \cdot uv^{3} + h_{3}(u,v) \cdot v^{6},$$ where the $h_{i}$ are power series and where $h_{1}$ has no constant term.  Now observe that there are no power series $$u(t),v(t) \in \k \llbracket t \rrbracket$$ with vanishing constant terms such that $$\ord g(u(t),v(t)) = 3,$$ a necessary condition for the germ $g$ to define a cusp. (Here the order $\ord$ of a power series with variable $t$ is the degree of the first non-zero term.) Thus the germ of a defining equation of $Q$ at $c$ cannot be contained in $\I_{\s(T,T^{*})}$, eliminating this curvilinear possibility. 

        \item[(c)] Hypothetically, if  $\s(T,T^{*}) \subset \Base (\Lambda)$ then our first claim is that $T$ must be entirely supported on the embedded point $b \in Q$:  Let $n \in Q$ be the node of $Q$ if it exists. By what is written immediately prior to the proof of part (a),  $T$ is supported somewhere in the set $\{b,n\}$. Let $T_{b}, T_{n}$ denote the connected components of $T$ supported on the respective points.  If $\length T_{n} \geq 2$, we argue as in part (a) to conclude that $\I_{\s(T,T^{*})}$ cannot contain a defining equation for $Q$.  If $\length T_{b} = 2$, then in local coordinates $(x,y)$ near $b$ the ideal $\I_{\s(T,T^{*})}$ is $(x^{2},xy^{2},y^{4})$. If  $f$ is a local defining equation of $Q$ near $b$, then the local ideal of $\Base (\Lambda)$ is given by $$\mathcal{J} := (xf,yf).$$  This ideal $\mathcal{J}$ has two linearly independent quadratic elements, modulo $(x,y)^{3}$, while $\I_{\s(T,T^{*})}$ does not.  Thus $\s(T,T^{*}) \nsubset \Base (\Lambda)$ in this case. 
        
        Therefore, we may assume that $T$ is entirely supported at the point $b$.  There are now two cases to investigate: either $T$ is a fat point, or $T$ is curvilinear. 
        
        In the fat case, $\I_{\s(T,T^*)}$ is contained in the cube of the maximal ideal $\m = (x,y) \subset \O_{\P^{2},b}$, while the two generators of $\mathcal{J}$ are not ($Q$ is smooth at $b$). So $\s(T,T^{*}) \not\subset \Base(\Lambda)$ in this case. 
        
        On the other hand, if $T$ is curvilinear a calculation as in the proof of part (b) implies that, modulo $\m^{3}$, $\I_{\s(T,T^{*})} \cap \m^{2}$ consists of a single quadratic form on $(\m/\m^{2})^{\vee}$ up to scaling. However, $\mathcal{J} \cap \m^{2}$ has two linearly independent quadratic forms modulo $\m^{3}$, and thus again $\s(T,T^{*}) \not\subset \Base (\Lambda)$. So our hypothetical situation is impossible, as we needed to show.

        \item[(d)] Let $n \in Q$ denote the node, let $x,y \in \O_{\P^{2},n}$ be local affine coordinates of $\P^{2}$ near $n$, and let $f(x,y) \in \O_{\P^{2},n}$ denote a local defining equation of $Q$.  If, hypothetically, $\s(T,T^{*}) \subset \Base (\Lambda)$, then $T$ must be entirely supported at $n$.  As in the proof of part (c), we consider the two possibilities: $T$ is either a fat point or is curvilinear with length $3$.  If $T$ is fat, then $\I_{\s(T,T^{*})}$ is contained in $\m^{3}$, where $\m = (x,y)$ is the maximal ideal. Furthermore, $\I_{\s(T,T^{*})}$ contains exactly one element (up to scale) in $\m^{3}$, modulo $\m^{4}$. However, the ideal $\mathcal{J} := (xf,yf)$ contains two $\k$-linearly independent such elements. Thus, $T$ cannot be fat.  

        If $T$ is curvilinear we observe that, modulo $\m^{4}$, the elements of $\I_{\s(T,T^{*})} \cap \m^{3}$ are binary cubic forms (on $(\m/\m^{2})^{\vee}$) all sharing a common factor which is a perfect square. This can be checked after passing to (any) analytic local coordinates $(u,v)$ -- we can choose convenient coordinates where $\I_{T} = (u,v^{3})$ and so $\I_{\s(T,T^{*})} = (u^2, uv^3,v^6)$.  Indeed, modulo $(u,v)^4$ all elements in $\I_{\s(t,T^{*})} \cap (u,v)^{3}$  are multiples of $u^{2}$.  This property of having a perfect square common factor is not shared by the ideal $\mathcal{J}$, and hence $\s(T,T^*) \not\subset \Base (\Lambda)$, eliminating this curvilinear possibility.
        \end{enumerate}
\end{proof}

\begin{proposition}
  \label{proposition:infline} If $(T,T^{*},\Lambda)$  is an element of $\Theta \cap \Inf$ then
  the fixed curve of $\Base(\Lambda)$ consists of a reduced
  line.
\end{proposition}

\begin{proof}
  We proceed by considering one by one the possible degrees of the
  fixed curve $C \subset \Base(\Lambda)$. Each case is resolved by the tension between 
  generality of $\Gamma_{13} = \{p_1, \dots, p_{13}\}$ (which by assumption is contained in $\Base(\Lambda)$) on the one hand, and the number of points in the intersection $\Gamma_{13} \cap C$ on the other. For ease of reading, let $\Gamma_{C} := \Gamma_{13} \cap C$.

\begin{enumerate}
  \item \emph{Assuming $\deg (C) = 4$},  the moving part of $\Lambda$ is a pencil of
  lines. Furthermore, $C$ must be reduced, as at most $5$ of the
  points of $\Gamma_{13}$ can lie on a conic. There are
  two possibilities for the number $\# \Gamma_{C}$: either $12$ or $13$, depending on whether the basepoint of $\Lambda'$ is in $\Gamma_{13}$ or not. Suppose first that $\# \Gamma_{C} = 13$. The unique, general, pencil of quartics
  determined by $\Gamma_{C}$ has finitely many singular members, each
  with a {\sl single} node located away from
  $\Gamma_{C} = \Gamma_{13}$. However, this prevents any length $9$
  scheme of the form $\s(T,T^{*})$ from being contained in $\Base(\Lambda)$, by parts (a), (c) and (d) of \Cref{lemma:QLNo}.  
  
  If
  $\# \Gamma_{C} = 12$, the argument is similar: the net of quartics
  determined by $\Gamma_{C}$ has singular members of only three types:
  (1) a nodal curve with a unique node, (2) a curve with exactly two
  nodes for singularities, and (3) a curve with a unique ordinary cusp.  Accordingly, $C$ is either smooth or has qualities (1), (2), or (3) just listed.  
  
  The thirteenth point not contained in $C$ must be the unique 
  basepoint of the moving part of $\Lambda$. So, by virtue of tangent space considerations, the only way
  for the length $9$ scheme $\s(T,T^{*})$ to be contained in
  $\Base(\Lambda)$ is if $\s(T,T^{*}) \subset C$. However, this is precluded by parts (a) and (b) of \Cref{lemma:QLNo}.

  \item \emph{Assuming $\deg (C) = 3$}, the moving part $\Lambda'$ of $\Lambda$ is a pencil of
  conics. As before, $C$ must be reduced, as $\Lambda'$ has at
  most $4$ of the points of $\Gamma_{13}$ in its base scheme, and so  the remaining points of $\Gamma_{13}$ cannot be contained in a line.
  In fact, this type of reasoning shows there is only one {\sl a priori}
  possibility: $\# \Gamma_{C} = 9$ and $C$ is the unique (smooth)
  cubic curve determined by $\Gamma_{C}$. The moving part $\Lambda'$ of $\Lambda$ is then a
  pencil of conics with base-scheme consisting of the four points 
  $\Gamma_{13} \setminus \Gamma_{C}$. The pencil $\Lambda$ 
  does not contain a subscheme of the form $\s(T,T^{*})$ in its base
  scheme $\Base(\Lambda)$ again by tangent space considerations, because $\Base(\Lambda)$ is the union of the smooth curve $C$ and the four reduced points $\Gamma_{13} \setminus \Gamma_{C}$.

  \item \emph{Assuming $\deg (C) = 2$}, then once again $C$ must be reduced as is seen by
  an argument similar to that found in the previous paragraph.  There
  are only two {\sl a priori} possibilities for $\# \Gamma_{C}$:
  $\# \Gamma_{C} = 4$ or $5$. Now, $\# \Gamma_{C}$ cannot be $4$,
  because then the moving part of $\Lambda$, a pencil of cubics, must
  contain all $9$ of the points $\Gamma_{13} \setminus \Gamma_{C}$ in
  its base scheme, contrary to the general nature of $\Gamma_{13}$.
  Thus $\# \Gamma_{C} = 5$ and $C$ is the unique smooth conic
  determined by $\Gamma_{C}$.  The remaining $8$ points of
  $\Gamma_{13} \setminus \Gamma_{C}$ define a general pencil of
  cubics, and this pencil of cubics is then the moving part $\Lambda'$ of $\Lambda$. $\Lambda'$ has
  a $9$-th basepoint not contained in $C$ (again because $\Gamma_{13}$ is
  general). Thus $\Base(\Lambda)$ is the smooth scheme consisting of
  the smooth conic $C$ and $9$ reduced points not contained in $C$. This base
  scheme evidently does not contain any subscheme of the form
  $\s(T,T^{*})$, again by tangent space considerations.
\end{enumerate}
  We've finished the analysis of all cases, and the proposition
  follows.
\end{proof}

\begin{proposition}
    \label{proposition:Movesingular} Suppose $(T,T^{*},\Lambda) \in \Theta \cap \Inf$.  Then the moving part $\Lambda'$ of $\Lambda$ is a pencil of quartics whose general member is a smooth quartic curve.
\end{proposition}

\begin{proof}
We argue by counting dimensions, keeping in mind that $\Gamma_{13} = (p_1, \dots, p_{13})$ is a general tuple. $\Lambda'$ must necessarily be a pencil of quartics thanks to \Cref{proposition:infline}.

Let $W$ denote the quasi-projective variety parametrizing pairs $(L,\Pi)$ where $L$ is a line in $\P^{2}$ and $\Pi$ is a pencil of quartics all sharing a singular point, and such that $\Base(\Pi)$ is finite. Such $\Pi$'s vary in a $22$-dimensional family, while a choice of $L$ provides $2$ more dimensions, and so $\dim W = 24$. 

For contradiction's sake, suppose $(T,T^{*},\Lambda) \in \Theta \cap \Inf$ is such that $\Base(\Lambda)$ has a line $L$ as fixed part and such that $\Pi := \Lambda'$, the moving part of $\Lambda$, satisfies $(L,\Pi) \in W$. By assumption, $\Gamma_{13}$ is contained in $\Base(\Lambda)$, and at most $2$ of the points of $\Gamma_{13}$ may lie on $L$ by generality of $\Gamma_{13}$.  On the other hand, at most $11$ of the points of $\Gamma_{13}$ may lie in the (finite) set $\Base(\Pi)$ because $\Pi$ varies in a $22$-dimensional family. Therefore, {\sl  exactly} $2$ points of $\Gamma_{13}$ must lie on $L$ and the remaining $11$ points of $\Gamma_{13}$ must lie inside $\Base(\Pi)$.  Furthermore, dimension considerations force $(L,\Pi)$ to be a general point of $W$. 

The contradiction will come from the fact that the base scheme $\Base (\Lambda)$ is the disjoint union of $L$ (without embedded points) and $\Base(\Pi)$. This is true because $(L,\Pi) \in W$ is general, and so  $\Base(\Pi)$ consists of one point of multiplicity $4$ along with $12$ other reduced points, none lying on $L$.  It is impossible for any scheme of the form $\s(T,T^{*})$ (which is everywhere non-reduced, has $2$-dimensional tangent space at all points, and has total length $9$) to be contained in $\Base(\Lambda) = L \cup \Base(\Pi)$, providing our contradiction.
\end{proof}

The next proposition is needed to remove an {\sl a priori} possible situation in $\Theta \cap \Inf$.

\begin{proposition}
\label{proposition:LinDom}
Suppose $(T,T^{*},\Lambda)$ is an element of  $\Dom(p) \cap \Lin(p)$ satisfying:
\begin{enumerate}
    \item The fixed part of $\Lambda$ is the line $L$ spanned by $T$, and
    \item $p \notin T$. 
\end{enumerate}
Then the moving part $\Lambda'$ of $\Lambda$ satisfies $$\length \left(\Base(\Lambda') \cap L\right) = 4.$$
\end{proposition}

\begin{proof}
    The proof uses a limit linear series style argument.  First observe that since the fixed part of $\Lambda$ is the reduced line $L$, the moving part $\Lambda'$ is a pencil of quartics and the length of $\Base(\Lambda')\cap L$ cannot exceed $4$.   We will approach the point $(T,T^{*},\Lambda)$ along a general $1$-parameter family contained in $\Dom(p)$. 
    
    So, let $(B,0)$ be a smooth, irreducible pointed affine curve, and let $\pi: \P^{2}_{B} \to B$ denote the natural projection.  Suppose $t_1, t_2, t_{3}$ are sections of $\pi$ satisfying: 
    \begin{enumerate}
        \item[(a)] For $b \neq 0$, the three points $t_{i}(b)$ form an honest triangle, denoted $\mathcal{T}_{b} \subset \P^{2}_{b}$.
        \item[(b)] The $b \to 0$ flat limit of $\mathcal{T}_{b}$, denoted $\mathcal{T}_{0}$, is thin. We let $L \subset \P^{2}_{b=0}$ denote the line containing $\mathcal{T}_{0}$.  ($\mathcal{T}_{0}$ is simply $T$ from the statement of the lemma, but we use calligraphic font for consistency in the following paragraphs.)
        \item[(c)] $\mathcal{T}_{0}$ does not contain the point $p$. (This is the second assumption in the proposition.)
        \item[(d)] The line $L$ contains $p$.
        \item[(e)] For $b \neq 0$, the four points $t_{1}(b), t_{2}(b), t_{3}(b),$ and $p$ are in linear general position.
    \end{enumerate}

    If $b \neq 0$, then the vector space $$V_{b} := H^{0}\left(\P^{2}_{b}, \I^{2}_{\mathcal{T}_{b}}\I_{p}(5)\right)$$ is $11$-dimensional because $\mathcal{T}_{b}$ and $p$ are in linear general position. Yet, when $b=0$, the vector space $V_{0}$ has dimension $12$.  Indeed, a quintic curve in $\P^{2}_{0}$ passing through $p$ and singular along $\mathcal{T}_{0}$ necessarily contains the line $L$ as an irreducible component and secondly its residual quartic curve $Q$ must contain $\mathcal{T}_{0}$.  The equations of such quartic curves form a $12$-dimensional vector space. Still, as $B$ is a smooth curve, the $b \to 0$ limit of $V_{b}$ is a well-defined $11$-dimensional vector space inside $V_{0}$ -- the following claim is equivalent to the conclusion of the proposition:

    \begin{claim}
    \label{claim:quinticforms}
        There exists a point $q \in L$ such that $$\lim_{b \to 0}V_{b} = \left\{\begin{tabular}{c} Quintic forms which are a product\\ $L \cdot Q$ where $Q$ is a quartic\\ containing the divisor $\mathcal{T}_{0} + q$ on $L$\end{tabular}\right\}.$$  
    \end{claim}

    Going forward we focus on \Cref{claim:quinticforms}.  We blow up $\P^{2}_{B}$  along the line $L \subset \P^{2}_{b=0}$. Let $$\widetilde{\P^{2}_{B}}$$ denote the blow up $\Bl_{L} \P^{2}_{B},$ and denote by $$\beta: \widetilde{\P^{2}_{B}} \to \P^{2}_{B}$$ the blowdown map.  The structural morphism $\pi: \widetilde{\P^{2}_{B}} \to B$ has fiber over $b \neq 0$ simply equal to $\P^{2}_{b}$, while the fiber $\pi^{-1}(\{0\})$ is the transverse union of two smooth surfaces: the exceptional divisor of $\beta$, denoted $E$,  and the proper transform of $\P^{2}_{0} \subset \P^{2}_{B}$ in the blow up, denoted $P$. $P$ maps isomorphically onto $\P^{2}_{0}$ under $\beta$. 

    The current geometric circumstance has certain features we want to emphasize:
\begin{itemize}
\item $\beta$ expresses $E$ as a $\P^{1}$-bundle over $L$. As such, $E$ is isomorphic to the Hirzebruch surface $\mathbb{F}_{1}$, with the intersection $E \cap P$ being the directrix curve which we denote by $D \subset E$. We let $F$ denote the divisor class of a fiber of the bundle $\beta|_{E}: E \to L$.

\item Since $L \cap \mathcal{T} = \mathcal{T}_{0} = \mathcal{T} \cap \P^{2}_{0}$ is a Cartier divisor on $\mathcal{T}$, it follows that $\mathcal{T}$ lifts isomorphically to $\widetilde{\P^{2}_{B}}$.  Let $\widetilde{\mathcal{T}} \subset \widetilde{\P^{2}_{B}}$ denote this lift, which is unique. (We've used here the assumption that $\mathcal{T}_{0}$ is thin.)  

\item $E \cap \widetilde{\mathcal{T}} = \widetilde{\mathcal{T}}_{0}$. Furthermore, $\widetilde{\mathcal{T}} \cap P = \varnothing$ because the sections $t_{i}$ comprising $\mathcal{T}$, by virtue of being sections, are not tangent to $\P^{2}_{0}$ from the start. 

\item Under $\beta$, the proper transform of the constant section $\{p\} \times B \subset \P^{2}_{B}$ becomes a section of $\pi$ which intersects $E$ at a point $\widetilde{p} \in E$ not on the directrix $D \subset E$.  

\item Clearly $\beta(\widetilde{p}) \notin \mathcal{T}_{0}$ because $p \notin \mathcal{T}_{0}$. And therefore $\widetilde{p} \notin \widetilde{\mathcal{T}}_{0}$.

\item Define $\mathcal{L}$ to be the invertible sheaf $\beta^{*}\O(5)(-E)$ on $\widetilde{\P_{B}^{2}}$. Then $\mathcal{L}|_{P} \simeq \O_{P}(4)$ and $\mathcal{L}|_{E} \simeq \O_{E}(D+5F)$.
\end{itemize}

Now we consider the finite, length 4 subscheme $Z := \widetilde{p} \cup \widetilde{\mathcal{T}}_{0} \subset E$.  On $E$, the term {\sl line} refers to any \emph{irreducible} curve in the linear series $|D+F|$ on $E$.  In particular, lines do not intersect the directrix $D$.

\begin{lemma}
    \label{lemma:D2F} Maintain the setting immediately prior. Then $$h^{0}(E,\I_{Z}(D+2F)) \geq 1$$ and $h^{0}(E,\I_{Z}(D+2F)) > 1$ if and only if $Z$ is contained in a line, in which case $h^{0}(E,\I_{Z}(D+2F)) = 2$.
\end{lemma}

\begin{proof}[Proof of \Cref{lemma:D2F}]
    This follows easily after translating the statement into a claim about length 5 subschemes in the plane imposing independent conditions on conics.  To get to this translation, simply blow down the Hirzebruch surface $E$ to a plane, contracting $D$.  We omit the details.
\end{proof}

At this point, for clarity of exposition, we make the following simplifying assumption:

\begin{assumption}
\label{assumption:Noline}
    Assume that $Z \subset E$ is {\bf not} contained in a line.
\end{assumption}
Operating under this assumption, we can finish the argument. On the reducible surface $P \cup E$, the global sections of 
$$\I_{\widetilde{\mathcal{T}}_{0}}^{2}\I_{\widetilde{p}}\otimes \mathcal{L}|_{P \cup E}$$ consist of the data of
\begin{itemize}
    \item A global section $Q$ of $\O_{P}(4)$, and
    \item a global section of $\O_{E}(D+5F)$ which, if nonzero, defines a reducible divisor of the form \[\beta^{-1}(\mathcal{T}_{0}) \cup C\] where $C$ is the unique curve in the linear series $|\I_{Z}(D+2F)|$ (\Cref{lemma:D2F})
\end{itemize}
satisfying the compatibility condition that both sections agree on the curve $D = E \cap P$.  Now we observe that under our simplifying \Cref{assumption:Noline}, the curve $C$ is irreducible and meets $D$ at a single point $q$.  Therefore, on $D$ we obtain a particular degree $4$ divisor, namely $\mathcal{T}_{0} + q$, which the quartic $Q$ is forced to contain in its zero scheme if $Q$ is to contribute to a global section of  $$\I_{\widetilde{\mathcal{T}}_{0}}^{2}\I_{\widetilde{p}}\otimes \mathcal{L}|_{P \cup E}.$$ Unwinding what this means before blowing up, we arrive at the conclusion of the theorem.

Finally, we will explain how to finish the proof of \Cref{proposition:LinDom} in the case where \Cref{assumption:Noline} does not hold. Suppose $Z \subset E$ is contained in a line, and call the line $L_1 \subset E$.  We then blow up the threefold $\widetilde{\P}^{2}_{B}$ along $L_1$.  Once again, $\widetilde{\mathcal{T}}$ will lift to the new blow up, as will the section which for general $b \in B$ selects the point $p$. Blowing up more and more in a similar fashion if necessary, after finitely many blow ups the strict transform of $\mathcal{T}$ and $\{p\} \times B$ are no longer collinear over $b=0$.  Then, we proceed as we did under the simplifying assumption, except now the special fiber is a chain of surfaces $E_0 \cup E_{1} \cup \dots \cup P$, the $E_{i}$'s being $\mathbb{F}_{1}$'s, glued one to the next along directrices on one side and lines on the other.  Then the line bundle $\beta^{*}\O(5)(-E_{0}-E_{1}-\dots E_{k})$ serves the same role as $\mathcal{L}$ did in the simplified situation, and the argument runs parallel to the simple case. 
\end{proof}

\begin{remark}
\label{remark:splitline}
    The conclusion of \Cref{proposition:LinDom} implies in particular that the pencil $\Lambda$ contains an element which is the union of the doubled line $2L$ with a cubic curve.
\end{remark}

\begin{proposition}
    \label{lemma:ThetaInf} Recall the set $\Theta$ from \Cref{equation:Theta}. Then $\Theta \cap \Inf = \varnothing.$
\end{proposition}

\begin{proof}
    Suppose for contradiction's sake that $(T,T^{*},\Lambda) \in \Theta \cap \Inf$. By \Cref{proposition:Movesingular}, $\Lambda$ is a pencil of the form $\left\{L \cdot Q_{t}\right\}, t\in \P^{1}$, where $L$ is a linear form (defining a line of the same name) and where $Q_{t}$ is a pencil of quartics \emph{with smooth general member}.  Considering tangent space dimensions, the condition $\s(T,T^{*}) \subset \Base (\Lambda)$ implies that $T$ is supported entirely on the set of embedded points $L \cap \Base \{Q_{t}\}$. 

    The first claim we make is that $T$ must in fact be a subscheme of $L$.  This is seen by checking all alternative possibilities -- we will investigate the most challenging case, and leave the rest to the reader.  Let $x,y$ denote affine coordinates, and consider the situation where  $T = V(x-y^{2},y^{3})$ and $L = V(x)$. We claim that if $x \cdot q(x,y) \in (x-y^{2},y^{3})^{2}$ then the constant and linear terms of $q(x,y)$ must both vanish.  It is clear that $q$ cannot have a constant term. To prove that $q$  has no linear term, it suffices to prove the same after applying the automorphism $x \mapsto x+y^{2},  y \mapsto y$.  Thus, we must show that if $$(x+y^{2}) \cdot h(x,y) \in (x^{2},xy^{3},y^{6})$$ then $h$ has no linear term. Let $h_{1},h_{2}, \dots$ be the linear, quadratic, etc... terms of $h$, so that $h = h_1 + h_2 + \dots$  First, by noting that the $y^{3}$-term of $(x+y^{2}) \cdot h$ must vanish, we conclude that $h_{1}$ must be a multiple of $x$. Write   $h_{1}=\lambda x$ for some $\lambda \in \k$.  Next, looking at the $xy^{2}$-term  we see that the $y^{2}$-term of $h_{2}$ must be $-\lambda y^{2}$.  Finally, by considering the $y^{4}$-term we also conclude that the $y^{2}$-term of $h_{2}$ must be $0$.  Thus $\lambda = 0$ and, as claimed, $h_{1}=0$.  The conclusion is that such a $T$ is eliminated from consideration from the case-by-case analysis because it would contradict the requirement that the general element of the pencil $Q_{t}$ is a \emph{smooth} quartic.  The remaining cases proceed in a similar fashion.

    We therefore assume $T \subset L$.  From the containment $\s(T,T^{*}) \subset \Base \{L \cdot Q_{t}\}$, it follows that $$T \subset \Base \{Q_{t}\}.$$ Now we bring in the assumption that the 13 \emph{general} points $p_{1} , \dots, p_{13}$ are contained in $\Base \Lambda$.  Not all $p_{i}$ can be contained in $\Base \{Q_{t}\}$.  Indeed, $Q_t$ would then be uniquely determined and a general pencil.  But it is not general: $\{Q_{t}\}$ contains a thin length $3$ scheme (namely $T$) in its base scheme.  It follows that at least one point, say $p_{13}$ is contained in $L \setminus \Base \{Q_{t}\}.$  Therefore, $(T,T^{*},\Lambda) \in \Dom(p_{13}) \cap \Lin(p_{13})$, and \Cref{proposition:LinDom} applies.  \Cref{remark:splitline} then says there is an element of the pencil $\{Q_{t}\}$ of the form $L \cdot C$ where $C$ is a cubic form.  However, then at least 11 of the remaining 12 points $p_{1}, \dots, p_{12}$ are contained in the cubic defined by $C$, contradicting generality of $\{p_{1}, \dots, p_{13}\}$.  The proposition follows. 
\end{proof}

\begin{theorem}
  \label{theorem:solutionset} Let $p_1, \dots, p_{13}$
  be 13 general points in $\P^{2}$, and let $\Theta$ be as in \Cref{equation:Theta}. Then   $$\Theta = \left\{ \begin{tabular}{c}
         $(T,T^{*}) \in\CT$ such that\\  $T$ is a singular triad  \\
         for $p_{1}, \dots, p_{13}$
    \end{tabular} \right\}.$$
\end{theorem}

\begin{proof}
    It is clear that the mentioned set of singular triads is contained in $\Theta$ -- the content of the theorem lies in the  reverse inclusion.  \Cref{lemma:ThetaInf} gives the inclusion $\Theta \subset \Fin$. Then the theorem immediately follows from \Cref{lemma:finitehonest} and B\'ezout's theorem.
\end{proof}

\section{Intersections in SQP}
\label{sec:finishing}

Let $p \in \P^{2}$ be a point and let $\ell \subset \P^{2}$ be a line. In this section we let $$\inc(p), \lin(p)$$ be the cycles on $\CT$ consisting of points $(T,T^{*})$ in $\CT$ where $T$ contains the point $p$ or where $T$ is collinear with $p$, respectively.  Observe that both cycles are pulled back from $\Hilb_{3}\P^{2}$ under the blowdown map $\CT \to \Hilb_{3}\P^{2}$. For brevity, we will write $\O(d)^{[3]}$ for the same-named tautological bundle on $\Hilb_{3}\P^{2}$, but pulled back to $\CT$.  Finally we let $H \subset \CT$ denote the divisor class of the locus of complete triangles $(T,T^{*})$ such that $T$ intersects the line $\ell$ non-trivially - $H$ is also evidently pulled back from $\Hilb^{3}\P^{2}$.

\begin{definition}
    Define $c_{i}(j) \in \CH^{i}\CT$ to be the Chern class $c_{i}\left(\O(j)^{[3]}\right)$. Define $e_k \in \CH^{k}\CT$ to be the Chern class $c_{k}(E)$, where $E$ is the rank $9$ vector bundle from \Cref{sec:Chern}.  Set 
    \begin{align}
        \Delta_{0}:= e_{3}^2-e_{2}e_{4},\\
        \Delta_{2}:= e_{2}^{2}-e_{1}e_{3}\\
        \Delta_{4}:= e_{1}^{2}-e_{2},\\
        \Delta_{6} := [\CT].
    \end{align}
    \end{definition}

    \begin{lemma}
        \label{lemma:Deltas}
        For each $i=0,1,2,3$ we have \[ \Delta_{2i} = \varphi_{*}\left( [\BP]^{13-i}\right).\]
    \end{lemma}

    \begin{proof}
        Let $i=0,1,2,$ or $3$. Fix $13-i$ general points on $\P^{2}$ and let $V$ denote the vector space of quintic forms vanishing at those $13-i$ points. Then the class $[\BP]^{13-i} \in \CH^{26-2i}\SQP$ can be represented by the cycle which parametrizes triples $(T,T^{*},\Lambda)$ where the  $\Lambda \subset V$. 

        Pushing this cycle down to $\CT$ via $\varphi$, we get the degeneracy scheme of the natural evaluation map of vector bundles \[ev: \underline{V} \to E\] consisting of points in $\CT$ where $ev$ has at least a $2$-dimensional kernel. The lemma follows from applying Porteous's formula to $ev$.
    \end{proof}

Recall the action of $\Gm$ on $\CT$ from \Cref{sec:Chern}. In order to use localization, we need to express the classes $[\lin(p)]$ and $[\inc(p)]$ as combinations of Chern classes of $\Gm$-equivariant bundles:

\begin{lemma}
    \label{lemma:IncLinclasses} In the Chow ring $\CH^{\bullet}\CT$, we have \begin{align}
        [\lin(p)] &= c_{2}(1),\\
        [\inc(p)] &= H^{2} - c_{2}(1) + 2c_{2}(2)-c_{2}(3)
        \end{align}
    \end{lemma}

\begin{proof}
    The first equation follows from the definition of the the second Chern class -- the subschemes lying in some member of the pencil of lines through $p$ are precisely those comprising the cycle $\lin(p)$.  The second equation is a consequence of general formulas for $c_{2}(d)$ for all $d$ found on page 93 of \cite{elencwajg2006explicit}.  In \emph{loc. cit.} the authors refer to $c_{2}(d)$ by the symbol $\mathcal{P}_{d}$.  We get the expression for $\inc(p)$ by combining the $d=2$ and $d=3$ cases of the formulas in \cite{elencwajg2006explicit}. (Observe that, although \cite{elencwajg2006explicit} concerns cycles on the Hilbert scheme $\Hilb_{3}\P^{2}$, we may pull them back to $\CT$ via the forgetful map $\CT \to \Hilb_{3}\P^{2}$.)
\end{proof}

%\begin{proposition}
  %\label{proposition:easyintersections}  The following hold in the Chow ring $\CH^{\bullet}(\SQP)$:
  %\begin{align}
  %[\Lin(p)]^3 &= 0,\\
  %[\Inc(p)]^3 &= [\varphi^{-1}((T,T^{*}))],\\
  %[\Inc(p)]^4 &= 0,\\
  %[\Inc(p)]^{2}\cap [\Lin(p)] &= 0 = [\Inc(p)] \cap [\Lin(p)]^2, \\
   %   \int_{\SQP}\left[ \BP(p) \right]^{13} &= 57728.
  %\end{align}
%\end{proposition}

\begin{proof}[Proof of \Cref{theorem:main}]
    By \Cref{theorem:solutionset} our task is to compute \[\int_{\SQP} [\Dom(p)]^{13}.\] We will do so by instead computing \[\int_{\CT}\varphi_{*}\left([\Dom(p)]^{13}\right),\] which by \Cref{theorem:BPdecomposition} equals \[\int_{\CT}\varphi_{*}\left(([\BP(p)] - 4[\Inc(p)] - [\Lin(p)])^{13}\right).\]  Using the push-pull formula for $\varphi$ and \Cref{lemma:Deltas}, we must then compute:  
\begin{align}
\label{eq:ultimateintegral}
\int_{\CT} \sum_{i=0}^{3} (-1)^{i}{13 \choose i} \cdot \Delta_{2i} \cdot \left( 4 [\inc(p)]+[\lin(p)]\right)^{i}.
\end{align}

By expressing all terms of \eqref{eq:ultimateintegral} in terms of the Chern classes $e_{k}$ of $E$ and using \Cref{lemma:IncLinclasses}, we then use the fixed-point analysis for $E$ in \S \ref{sec:Chern} and the calculations in \Cref{proposition:weightstautological} and evaluate the Atiyah-Bott localization expression.  We perform this computation in \S \ref{sec:sagecode}.

\end{proof}

\section{Lingering questions}
\label{sec:questions}

Countless questions remain unanswered -- we highlight some below.

\begin{question}
    \label{question:association}
    Association's presence in all known cases of the Veronese counting problem is hard to ignore -- what is its role in the general problem? 
\end{question}  

In fact, a closer look shows an intriguing possibility, which we now explain.  In all known instances, it appears as though association is a composite of a Cremona transformation followed by a Veronese embedding.  Specifically, let $m = {n+d \choose d} + n +1$, and consider a general tuple of points $(p_1, \dots, p_{m}) \in (\P^{n})^{m}$.  Let $(q_1, \dots, q_{m}) \in (\P^{N})^{m}$ be associated to $(p_{1}, \dots, p_{m})$, where $N = {n+d \choose d}-1$. Finally let $\ver_{d} :\P^{n} \to \P^{N}$ denote the standard $d$-uple Veronese embedding (after choosing coordinates).  

In every understood case of the general Veronese counting problem, there exists a Cremona transformation $\gamma: \P^{n} \dashrightarrow \P^{n}$ and an automorphism $g: \P^{N} \to \P^{N}$ such that the composite $g \circ \ver_{d} \circ \gamma$ sends each point $p_{i}$ to its corresponding point $q_{i}$. The Cremona transformation $\gamma$ need not be unique, but the fact that it exists in the first place is not obvious to us.  And so a follow-up to \Cref{question:association} would be to determine whether a ``$g \circ \ver_{d} \circ \gamma$'' mapping always exists which outputs an associated set for $p_1, \dots, p_m$ in every instance of the enumerative problem.

\begin{remark}
    In this direction, a dimension count suggests that for a fixed set of $14$ general points $a_1, \dots, a_{14}$ in $\P^{3}$, there should exist finitely many \emph{quadro-quartic} Cremona transformations $\gamma: \P^{3} \dashrightarrow \P^{3}$ such that the composition of $\gamma$ with a $2$-uple Veronese embedding sends the tuple $(a_{1}, \dots, a_{14})$ to an associated tuple $(b_{1}, \dots, b_{14}) \in (\P^{9})^{14}$.  This may hint at an approach, admittedly ambitious, to finding the number $\nu_{2,3}$ of $2$-Veronese $3$-folds through $14$ general points.
\end{remark}   

\begin{question}
    \label{question:numerical}
    Is it possible to confirm Coble's example using numerical algebraic geometry software like \texttt{homotopycontinuation.jl}?
\end{question}

\begin{question}
    \label{question:monodromy} It can be shown that the monodromy group of Coble's enumerative problem $\nu_{2,2}=4$ is the full symmetric group $S_{4}$.  Is the monodromy group for $\nu_{3,2}$ also the full symmetric group?
\end{question}

\begin{question}
    \label{question:degeneration} Is there a straightforward degeneration argument reproving $\nu_{2,2}=4$?
\end{question}

\begin{question}
    \label{question:charp} Association is a characteristic independent theory.  By following Coble's reasoning in characteristic $p$, we find that when $p=2$ and only when $p=2$, the number $\nu_{2,2}$ drops -- it changes from $4$ to $2$ because the $2$-torsion of a general elliptic curve in characteristic $2$ consists of only $2$ points.  Modulo which primes $p$ (if any) does our $\nu_{3,2}= 4246$ calculation change?
\end{question}

\begin{question}
    \label{question:integral} Are there other integral expressions of the form \[\nu_{d,n} = \int_{X} \alpha^{m}\] where $X$ is a relatively tractable smooth projective variety and $\alpha$ is some cycle on $X$? 
\end{question}

\section{Sage code}
\label{sec:sagecode}

We provide the sage code we used to compute the number $\nu_{3,2}$, as well as several checks we performed to acquire confidence that there are no mistakes in the computation of weights of relevant bundles at fixed points. One thing to note is that we plug in specific values for $a,b,c$ into the equivariant expressions arising in the Atiyah-Bott localization formula. These are secretly constant, so the reader can verify that by changing the inputs of $a,b,c$, the various calculations below do not change. This is a further check on the integrity of the calculations.

\begin{lstlisting}[basicstyle=\ttfamily\tiny]  
  var('a','b','c')

  def symmetrize(p):
      return p(a=a,b=b,c=c) + p(a=a,b=c,c=b) + p(a=b,b=a,c=c) + p(a=b,b=c,c=a) + p(a=c,b=a,c=b) + p(a=c,b=b,c=a)
  
  # The ith elementary symmetric polynomial
  def sigma(i, L):
      ind = Set(range(0,len(L)))
      
      return sum([ prod([L[x] for x in S]) for S in ind.subsets(i) ])
  
  # BUNDLE WEIGHTS AT FIXED POINTS:
  # The first entry corresponds to an honest triangle, which is unchanged under the action of S_3 permuting homogeneous coordinates, and the rest have orbits of size 6.
  
  
  # Weight data for the bundle E on CT.
  E = [
      (5*a, 5*b, 5*c, 4*a+b, 4*a+c, 4*b+c, 4*b+a, 4*c+a, 4*c+b),
      (5*a, 4*a+b, 4*a+c, 5*c, 4*c+a, 4*c+b, 3*c+a+b, 3*c+2*b, 2*c+3*b),
      (5*b, 4*b+a, 4*b+c, 5*c, 4*c+a, 4*c+b, 3*c+a+b, 3*c+2*b, 2*c+3*b),
      (5*c, 4*c+b, 3*c+2*b, 2*c+3*b, c+4*b, 5*b, 4*c+a, 3*c+a+b, 2*c+a+2*b),
      (5*c, 4*c+a, 3*c+2*a, 4*c+b, 3*c+a+b, 2*c+2*a+b, 3*c+2*b, 2*c+2*b+a, 2*c+3*b),
      (5*c, 4*c+a, 3*c+2*a, 2*c+3*a, 4*c+b, 3*c+a+b, 3*c+2*b, 2*c+2*b+a, 2*c+3*b)
      ]
  
  # Weight data for the tangent bundle T of CT
  T = [
      (c-a, c-b, b-c, b-a, a-b, a-c),
      (a-c, a-b, c-a, 2*c-2*b, b-a, c-b),
      (c-a, 2*c-2*b, b-a, c-b, b-c, b-a),
      (c-a, 3*c-3*b, b-a, 2*c-2*b, 2*b-c-a, c-b),
      (3*b-3*a, 2*b-2*a, b-a, c-a, c-b, a-2*b+c),
      (a-b, c-b, c-a, 2*b-2*a, c-b, b-a)
      ]
  
  # Some Chern class expressions of E at the six representative fixed points.
  c1E = [ expand(sigma(1, x)) for x in E ]
  c2E = [ expand(sigma(2, x)) for x in E ]
  c3E = [ expand(sigma(3, x)) for x in E ]
  c4E = [ expand(sigma(4, x)) for x in E ]
  c6T = [ expand(sigma(6, x)) for x in T ]
  
  # This is the Atiyah-Bott localization formula for computing -c2(E) * c4(E) + c3(E)^2. It is the wrong answer, due to excess. 
  Wrong = expand(-c2E[0]*c4E[0] + c3E[0]^2)/expand(c6T[0]) + sum([symmetrize(expand(-c2E[i]*c4E[i] + c3E[i]^2)/expand(c6T[i])) for i in range(1,6)])
  
  # O's tautological rank three bundle's weights at the six fixed representative fixed points
  OOO = [
      (0, 0, 0),
      (0, b-c, 0),
      (0, b-c, 0),
      (0, b-c, 2*b-2*c),
      (0, a-c, b-c),
      (0, a-c, b-c)
      ]
  
  c1OOO = [ expand(sigma(1, x)) for x in OOO ]
  c2OOO = [ expand(sigma(2, x)) for x in OOO ]
  c3OOO = [ expand(sigma(3, x)) for x in OOO ]
  
  # O(1)'s tautological rank 3 bundle's weights at the 6 representative fixed points, followed by its Chern classes at those points.
  Oone = [
      (a, b, c),
      (a, b, c),
      (b, b, c),
      (c, b, 2*b-c),
      (a, b, c),
      (a, b, c)
      ]
  
  c1Oone = [ expand(sigma(1, x)) for x in Oone ]
  c2Oone = [ expand(sigma(2, x)) for x in Oone ]
  c3Oone = [ expand(sigma(3, x)) for x in Oone ]
  
  # O(2)'s tautological rank 3 bundle's weights at the six representative fixed points, followed by its Chern classes at those points.
  Otwo = [
      (2*a, 2*b, 2*c),
      (2*a, b+c, 2*c),
      (2*b, b+c, 2*c),
      (2*b, b+c, 2*c),
      (a+c, b+c, 2*c),
      (a+c, b+c, 2*c)
      ]
  
  c1Otwo = [ expand(sigma(1, x)) for x in Otwo ]
  c2Otwo = [ expand(sigma(2, x)) for x in Otwo ]
  c3Otwo = [ expand(sigma(3, x)) for x in Otwo ]
  
  # O(3)'s tautological bundles weights at the six representative fixed points, followed by its Chern classes at those points.
  Othree = [
      (3*a, 3*b, 3*c),
      (3*a, b+2*c, 3*c),
      (3*b, b+2*c, 3*c),
      (b+2*c, 2*b+c, 3*c),
      (a+2*c, b+2*c, 3*c),
      (a+2*c, b+2*c, 3*c)
      ]
  
  c1Othree = [ expand(sigma(1, x)) for x in Othree ]
  c2Othree = [ expand(sigma(2, x)) for x in Othree ]
  c3Othree = [ expand(sigma(3, x)) for x in Othree ]
  
  # The line bundle H's weights at the six representative fixed points
  H = [
      a+b+c,
      a+2*c,
      b+2*c,
      3*c,
      3*c,
      3*c
      ]
  # We will just use the symbol H for its first Chern class c1H
  
  # Delta_0's expression at the six representative fixed points
  Delta0 = [ c3E[i]^2-c2E[i]*c4E[i] for i in range(0,6)]
  
  
  # Delta_2's expression at the six representative fixed points
  Delta2 = [ c2E[i]^2-c1E[i]*c3E[i] for i in range (0,6)]
  
  # Delta_4's expression at the six representative fixed points
  Delta4 = [ c1E[i]^2-c2E[i] for i in range(0,6)]
  
  
  
  # THE CLASSES INC, LIN, ETC:
  # The class Inc to be used in localization formula (tuple of 6 degree two expressions in abc)
  
  Inc = [H[i]^2 - c2Oone[i] + 2*c2Otwo[i] - c2Othree[i] for i in range(0,6)]
  
  # the class of Lin to be used in localization formula (tuple of 6 degre two expressions in abc)
  
  Lin = [c2Oone[i] for i in range(0,6)]
  
  # The class 4Inc + Lin. This is relevant because of the relation BP(p) = Dom(p) + 4Inc(p) + Lin(p).
  
  FourIncPlusLin = [4*Inc[i] + Lin[i] for i in range(0,6)]
  
  # The expression we wish to integrate, an expression in a,b,c at each of the six representative fixed points.
  
  INTEGRAND = [Delta0[i] - 13*Delta2[i]*FourIncPlusLin[i] + 78*Delta4[i]*FourIncPlusLin[i]^2 - 286*FourIncPlusLin[i]^3 for i in range(0,6)]
  
  # Application of Atiyah-Bott to integrate INTEGRAND, summing over all 31 fixed points, remembering that the honest triangle is its own S3 orbit, so we do not symmetrize it.  
  
  Answer = expand(INTEGRAND[0])/expand(c6T[0]) + sum([symmetrize(expand(INTEGRAND[i])/expand(c6T[i])) for i in range(1,6)])
  
  
  
  # TESTS: 
  # Let's compute H^6. This should be 15.
  Hsixth = expand(H[0]^6)/expand(c6T[0]) + sum([symmetrize(expand(H[i]^6)/expand(c6T[i])) for i in range(1,6)])
  
  #print(Hsixth(a=2, b=5, c=-9)) This indeed yields 15.
  
  # Let's compute H^4*Inc. This should be 3. 
  HfourthInc = expand(Inc[0]*H[0]^4)/expand(c6T[0]) + sum([symmetrize(expand(Inc[i]*H[i]^4)/expand(c6T[i])) for i in range(1,6)])
  
  # print(HfourthInc(a=2,b=3,c=-2)) #This indeed gives 3.
  
  # Let's compute c_{3}(2)^2, which should be (4 choose 3) = 4. 
  Otwocheck = expand(c3Otwo[0]^2)/expand(c6T[0]) + sum([symmetrize(expand(c3Otwo[i]^2)/expand(c6T[i])) for i in range(1,6)])
  
  # print(Otwocheck(a=5,b=2,c=7)) #Indeed this gives 4.
  
  # Let's compute c_3(3)^{2}, which should be (9 choose 3).
  Othreecheck = expand(c3Othree[0]^2)/expand(c6T[0]) + sum([symmetrize(expand(c3Othree[i]^2)/expand(c6T[i])) for i in range(1,6)])
  
  # print(Othreecheck(a=-3, b=5, c=4)) #This indeed gives (9 choose 3) = 84.
  
  # Let's compute c_3(2)*c_3(3), which should be (6 choose 3) = 20.
  Otwothreecheck = expand(c3Othree[0]*c3Otwo[0])/expand(c6T[0]) + sum([symmetrize(expand(c3Othree[i]*c3Otwo[i])/expand(c6T[i])) for i in range(1,6)])
  
  # print(Otwothreecheck(a=3,b=2,c=-5)) #Indeed this gives 20.
  
  Lincubed = expand(Lin[0]^3)/expand(c6T[0]) + sum([symmetrize(expand(Lin[i]^3)/expand(c6T[i])) for i in range(1,6)])
  
  # print(Lincubed(a=-2,b=13,c=4)) yields 0 as is should because Lin cubed is indeed 0
  
  IncsquaredLin = expand(Inc[0]^2 * Lin[0])/expand(c6T[0]) + sum([symmetrize(expand(Inc[i]^2 * Lin[i])/expand(c6T[i])) for i in range(1,6)])
  
  # print(IncsquaredLin(a=-1,b=12,c=13)) yields 0, which it should. 
  
  Inccubed = expand(Inc[0]^3)/expand(c6T[0]) + sum([symmetrize(expand(Inc[i]^3)/expand(c6T[i])) for i in range(1,6)])
  
  # print(Inccubed(a=12,b=-3,c=5)) yields 1, as it should because Inc^3 = 1 for simple geometric reasons.
  
  #FINAL COMPUTATIONS:
  print(Wrong(a=45,b=3,c=10)) # This is the wrong answer 57728, where we apply the Porteous formula to the bundle E.
  print(Answer(a=-20,b=9,c=7)) #This gives 4246.
\end{lstlisting}

\bibliographystyle{alpha}
 \bibliography{bibliography}

\end{document}